\documentclass[11pt,a4paper]{article}

\usepackage{eepic}

\usepackage{amsmath}
\usepackage{amssymb}
\usepackage{amsthm}
\usepackage{latexsym}
\usepackage{pstricks}
\usepackage{amsbsy}
\usepackage{amscd}
\usepackage{mathrsfs}
\usepackage{tipa}
\usepackage{txfonts}
\usepackage{amsfonts}
\usepackage{graphicx}
\usepackage{tabularx}
\usepackage{listings}
\usepackage{tikz}
\usepackage{hyperref}
\tikzstyle{dot} = [inner sep=0pt,thick,fill=black,circle,minimum size=2.5pt]
\tikzstyle{line} = [draw, -latex]
\newtheorem{thm}{Theorem}[section]
\newtheorem{cor}[thm]{Corollary}
\newtheorem{lem}[thm]{Lemma}

\newtheorem{exm}{Example}
\newtheorem{prop}[thm]{Proposition}
\newtheorem{defn}[thm]{Definition}
\newtheorem{rem}[thm]{Remark}
\newtheorem{defn-prop}[thm]{Definition-Proposition}

\newcommand{\qihao}{\fontsize{7.25pt}{\baselineskip}\selectfont}

\usepackage[all,poly]{xy}
\usepackage[all]{xy}
\usepackage{xypic}

\begin{document}

\begin{center}
{\Large \bf Cluster automorphism groups of cluster algebras with coefficients
\footnote{Supported by the NSF of China (Grants 11131001)}}

\bigskip

{\large Wen Chang and
 Bin Zhu}

{\small
\begin{tabular}{cc}
Department of Mathematical Sciences & Department of Mathematical
Sciences
\\
Tsinghua University & Tsinghua University
\\
  100084 Beijing, P. R. China &   100084 Beijing, P. R. China
\\
{\footnotesize E-mail: changw12@mails.tsinghua.edu.cn} &
{\footnotesize E-mail: bzhu@math.tsinghua.edu.cn}
\end{tabular}
}
\bigskip


\end{center}

\begin{abstract}
We study the cluster automorphism group of a skew-symmetric cluster algebra with geometric coefficients. For this, we introduce the notion of gluing free
cluster algebra, and show that under a weak condition the cluster automorphism group of a gluing free cluster algebra is a subgroup of the
cluster automorphism group of its principal part cluster algebra (i.e. the corresponding cluster algebra without coefficients). We show that several classes of cluster algebras with coefficients are gluing
free, for example, cluster algebras with principal coefficients, cluster algebras with universal geometric coefficients, and cluster
algebras from surfaces (except a 4-gon) with coefficients from boundaries. Moreover, except four kinds of surfaces, the cluster
automorphism group of a cluster algebra from a surface with coefficients from boundaries is isomorphic to the cluster automorphism group of
its principal part cluster algebra; for a cluster algebra with principal coefficients, its cluster automorphism group is isomorphic to the automorphism group of its initial quiver.
\end{abstract}

\def\s{\stackrel}
\def\Longrightarrow{{\longrightarrow}}
\def\A{\mathcal{A}}
\def\B{\mathcal{B}}
\def\C{\mathcal{C}}
\def\D{\mathcal{D}}
\def\F{\mathcal{F}}
\def\T{\mathcal{T}}
\def\R{\mathcal{R}}
\def\P{\mathcal{P}}
\def\S{\Sigma}
\def\H{\mathcal{H}}
\def\U{\mathscr{U}}
\def\V{\mathscr{V}}
\def\M{{\bf{Mut}}}
\def\L{\mathcal{L}}
\def\W{\mathscr{W}}
\def\X{\mathscr{X}}
\def\Y{\mathscr{Y}}
\def\x{{\mathbf x}}
\def\ex{{\mathbf{ex}}}
\def\fx{{\mathbf{fx}}}
\def\I{\mathcal {I}}
\def\add{\mbox{add}}
\def\Aut{\mbox{Aut}}
\def\coker{\mbox{coker}}
\def\deg{\mbox{deg}}
\def\diag{\mbox{diag}}
\def\dim{\mbox{dim}}
\def\End{\mbox{End}}
\def\Ext{\mbox{Ext}}
\def\Hom{\mbox{Hom}}
\def\Gr{\mbox{Gr}}
\def\id{\mbox{id}}
\def\Im{\mbox{Im}}
\def\ind{\mbox{ind}}
\def\mod{\mbox{mod}}
\def\mul{\multiput}
\def\c{\circ}
\def \text{\mbox}
\newcommand{\homeo}{\textup{Homeo}^+(S,M)}
\newcommand{\homeoo}{\textup{Homeo}_0(S,M)}
\newcommand{\mg}{\mathcal{MG}(S,M)}
\newcommand{\mmg}{\mathcal{MG}_{\bowtie}(S,M)}

\newcommand{\Z}{\mathbb{Z}}
\newcommand{\Q}{\mathbb{Q}}
\newcommand{\N}{\mathbb{N}}

\def\MM#1{{\bf{(CM#1)}}}

\hyphenation{ap-pro-xi-ma-tion}

\medskip



\section{Introduction}

After introduced by Sergey Fomin and Andrei Zelevinsky in \cite{FZ02}, cluster algebras have been showed to be linked to various areas of mathematics, see for examples, \cite{GLS08, F10, L10, R10, K12, M14}, and so on. However, as an algebra itself with combinatorical structure, it is natural and interesting to study the symmetries of a cluster algebra. For this, Assem, Schiffler and Shramchenko \cite{ASS12} introduced cluster automorphisms and the cluster automorphism group of a cluster algebra without coefficients. These concepts and some similar ones are studied in many papers \cite{S10,ASS12,BQ12,ASS13,BD13,KP13,N13,Z06}. In
this paper, we initial the study of the cluster automorphism group of a cluster algebra with coefficients. We are interested in the dependence of the cluster automorphism group on the choice of coefficients. We are also interested in the relations between the cluster automorphism groups of different cluster algebras.\\

We consider in this paper the skew-symmetric cluster algebras of geometric type, that is, the cluster algebras determined by ice quivers without loops nor $2$-cycles. An ice quiver is an oriented diagram $Q$ associated a subset $F=\{n+1, n+2, \cdot \cdot \cdot, n+m\}$ (the set of frozen vertices) of its vertex set $Q_0=\{1, 2, \cdot \cdot \cdot, n+m\}$. The full subquiver of $Q$ with vertex set $Q_0\setminus F$ (the set of exchangeable vertices) is called the principal part of $Q$, and we denote it by $Q^{ex}$. By associating each vertex $1\leqslant i \leqslant n+m$ of $Q$ an indeterminate element $x_i$, we have a set $\x=\{x_1, x_2, \cdot \cdot \cdot, x_{n+m}\}$, which is called a cluster. Then the cluster algebra $\A_Q$ is a $\Z-$subalgebra of the rational function field $\F=\Q(x_1,\cdot \cdot \cdot, x_{n+m})$ generated by variables in clusters obtained by iterated operations so called mutations from the initial cluster $\x$. The variables labeled by frozen vertices are the coefficients of $\A_Q$. The cluster algebra $\A_{Q^{ex}}$ is called the principal part cluster algebra of $\A_Q$, it is a coefficient free cluster algebra.\\

A cluster automorphism of $\A_Q$ is an algebra automorphism which maps the initial cluster $\x$ to a cluster and commutes with the mutation. An easy observation is that the group $Aut(\A_Q)$ depends on the choice of coefficients. To describe this more precisely, we classify the coefficients of $\A_Q$ by considering the relations of frozen vertices in $Q$ with
respect to the exchangeable vertices. Let $j$ and $k$ be two frozen vertices, if the number of arrows from $j$ to $i$ (or from $i$ to $j$) is equal to the number of arrows from $k$ to $i$ (or from $i$ to $k$) for any exchangeable vertex $i$, then we say that these two vertices $j$ and $k$ are strictly glueable (Definition \ref{definition gluing free cluster algebras}). It is not hard to see that exchanging the coefficients $x_j$ and $x_k$ induces a cluster automorphism of $\A_Q$ (see Proposition \ref{Prop of CA}). The cluster algebra $\A_Q$ is called gluing free if there exist no frozen vertices are strictly glueable in $Q$ (see Definition \ref{def: gluing free cluster algebra}). Then we prove that for a gluing free cluster algebra $\A_Q$ with at least two exchangeable vertices in $Q$, the group $Aut(\A_Q)$ is a subgroup of $Aut(\A_{Q^{ex}})$ (see Theorem \ref{CCA subgroup of CA}). \\

To prove Theorem \ref{CCA subgroup of CA}, we consider the exchange graph $E_Q$ of $\A_Q$, which is a n-regular connected graph whose vertices are clusters and whose edges are labeled by mutations. We introduce the automorphism group $Aut(E_Q)$ of $E_Q$ which consisting of automorphisms as a graph\footnote{We owe this definition to Thomas Br\"ustle.}. Then by using the fact that $E_Q$ is independent on the choice of coefficients\cite{CKLP13}, we prove Theorem \ref{CCA subgroup of CA} by viewing $Aut(\A_Q)$ and $Aut(\A_{Q^{ex}})$ as subgroups of $Aut(E_Q)$.\\

In this paper, three kinds of cluster algebras are of particular importance. The first one is the universal geometric cluster algebra
(Definition \ref{definition coefficient specilazation}), which is slightly different from the initial universal cluster algebra introduced in
\cite{FZ07}. It is a universal object, in the view point of coefficient specialization, in the set of cluster algebras with the same principal
part. The second one is the cluster algebra arising from an oriented marked Reimman surface with boundary\cite{FST08}. The another kind is a cluster algebra with principal coefficients. We show that a cluster algebra with universal geometric coefficients, a cluster algebra (with coefficients) from a surface (except a 4-gon), and a cluster algebra with principal coefficients are gluing free in Proposition \ref{pgf of UCA}, Proposition \ref{prop of SM pgf}, and Proposition \ref{pgf of PCA} respectively. Thus by Theorem \ref{CCA subgroup of CA}, the automorphism groups of these cluster algebras are all the subgroups of the corresponding principal part cluster algebras.\\

Generally, for a gluing free cluster algebra $\A_{Q}$, $Aut(\A_{Q})$ may be a proper subgroup of $Aut(\A_{Q^{ex}})$ (see Example \ref{universal subgroup}). But for the cluster algebras from oriented marked Riemann surfaces, these two groups are isomorphic with each other (Theorem \ref{theorem SCA}).
For a cluster algebra with principal coefficients, the cluster automorphism group is isomorphic to the automorphism group of the initial quiver (Theorem \ref{thm:principal coefficients}).\\

As showed in Example \ref{universal subgroup}, for an universal geometric cluster algebra $\A^{univ}$ with principal part $\A$, $Aut(\A^{univ})$ may be a proper subgroup of $Aut(\A)$. However, we prove in the subsequent paper \cite{CZ15a} that $Aut(\A^{univ})$ is always isomorphic to $Aut(\A)$, if $\A$ is a cluster algebra of finite type and $\A^{univ}$ is the $FZ$-universal cluster algebra introduced in \cite{FZ07}. Generally, $Aut(\A_{Q})$ may be a proper subgroup of $Aut(E_{Q})$, even though $\A_{Q}$ is coefficient free. In the subsequent paper \cite{CZb15} we prove that for a coefficient free cluster algebra $\A_Q$, the two groups $Aut(\A_{Q})$ and $Aut(E_{Q})$ are isomorphic, if $\mathcal{A}$ is of finite type, excepting types of rank two and type $F_4$, or $\mathcal{A}$ is of skew-symmetric finite mutation type, that is, a cluster algebra with only finite quivers up to isomorphisms.\\

The paper is organized as follows. In section \ref{section Preliminaries}, we recall some basic notions on cluster algebras, especially cluster
automorphisms, coefficient specializations and cluster algebras arising from surfaces. In subsection \ref{section gluing free cluster algebras} we define the notion of gluing free cluster algebra and prove some properties. We introduce and consider the automorphism group of an exchange graph in subsection \ref{section exchange graph}. Finally, in subsection \ref{section main results} we prove some main results.

\section{Preliminaries}\label{section Preliminaries}

\subsection{Cluster algebras}\label{section cluster algebras}
We recall that a {\it quiver} is a quadruple $(Q_0,Q_1,s,t)$ consisting of a set of {\it vertices} $Q_0$, of a set of {\it arrows} $Q_1$, and
of two maps $s, t$ which map each arrow $\alpha \in Q_1$ to its {\it source} $s(\alpha)$ and its {\it target} $t(\alpha)$, respectively. An
{\it ice quiver} is a quiver $Q$ associated a subset $F$ (the set of {\it frozen vertices}) of $Q_0$. The full subquiver $Q^{ex}$ of
$Q$ with vertex set $Q_0\setminus F$ (the set of {\it exchangeable vertices}) is called the {\it principal part} of $Q$. The opposite quiver of $Q$ is a quiver $Q^{op}$ obtained from $Q$ by reversing all the arrows. We always assume in this
paper that there are no loops nor 2-cycles in an ice quiver, and no arrows between frozen vertices. We also assume that an ice quiver and its principal part are
connected. For two ice quivers $Q$ and $Q'$, an {\it isomorphism} $\sigma$ from $Q$ to $Q'$ is a bijection from $Q_0$ to $Q'_0$, which maps exchangeable vertices (and frozen vertices) to exchangeable ones (and frozen ones), such that the number of arrows from a vertex $i$ to a vertex $j$ in $Q_0$ is equal to the number of arrows from $\sigma(i)$ to $\sigma(j)$ in $Q'_0$, we write $Q\cong Q'$ to quivers which are isomorphic with each other. Then for a quiver $Q$, all of its automorphisms consist a group $Aut(Q)$, we call it the {\it automorphism group} of $Q$.\\

Let $n+m=|Q_0|$ be the number of vertices in $Q$, and denote the vertices by $Q_0=\{1, 2,
\cdot \cdot \cdot, n+m\}$ and the frozen vertices by $F=\{n+1, n+2, \cdot \cdot \cdot, n+m\}$. We associate an extended skew-symmetric matrix
$B=(b_{ji})_{(n+m)\times n}$ to $Q$, where
$$b_{ji}=\left\{\begin{array}{ll}
						\sharp\{\alpha: j \to i \text{~in~} Q\} & \textrm{ if there exist arrows from j to i in Q}~; \\
                        -\sharp\{\alpha: i \to j \text{~in~} Q\} & \textrm{ if there exist arrows from i to j in Q}~; \\
						0 & \textrm{ if there exist no arrows between i and j in Q}~. \\
					\end{array}\right.$$
We call $B$ the {\it exchange matrix} of $Q$, its upper $n\times n$ part the {\it principal part} of $B$ and its lower $m\times n$ part the {\it frozen part} of
$B$. Let $i$ be an exchangeable vertex of $Q$, we define a mutation of $Q$ and a mutation of $B$ as follows.\\

\begin{defn}\cite{FZ02}\label{def: quiver mutation}
\begin{itemize}
\item[(a)]The mutation of $Q$ at $i$ is an ice quiver $\mu_i(Q, F)=(\mu_i(Q), \mu_i(F))$, where $\mu_i(F)=F$ and $\mu_i(Q)$ is obtained from $Q$ by:
\begin{itemize}
\item inserting a new arrow $\gamma: j\to k$ for each path $j\s{\alpha}\rightarrow i \s{\beta}\rightarrow k$;
\item inverting all arrows passing through $i$;
\item removing the arrows in a maximal set of pairwise disjoint $2$-cycles;
\item removing the arrows between frozen vertices.
\end{itemize}
\item[(b)] The mutation of $B$ at $i$ is a matrix $\mu_i(B)=(b'_{jk})_{n\times n} \in M_{n\times n}(\Z)$ given by
					$$b'_{jk} = \left\{\begin{array}{ll}
						- b_{jk} & \textrm{ if } k=i \textrm{ or } j=i~; \\
						b_{jk} + \frac 12 (|b_{jk}|b_{ki} + b_{jk}|b_{ki}|) & \textrm{ otherwise.}
					\end{array}\right.$$
\end{itemize}
\end{defn}
Then these two kinds of mutations correspond with each other, that is, $\mu_i(B)$ is the exchange matrix of $\mu_i(Q)$. We always write $\mu_i(Q)$ to $\mu_i(Q, F)$ for brevity. By associating each vertex $1\leqslant i \leqslant n+m$ in $Q_0$ an indeterminate element $x_i$, we have a set $\x=\{x_1, x_2, \cdot \cdot \cdot, x_{n+m}\}$. We call $\S=(Q,\x)$ a seed, $\x$ a cluster and an element in $\x$ a cluster variable.
\begin{defn}\cite{FZ02}\label{def: seed mutation}
Let $i$ be an exchangeable vertex of $Q$, the mutation of the seed $\S$ in the direction $i$ (or $x_i$) is a seed $\mu_i(\S)=(\mu_i(Q),\mu_i(\x))$, where $\mu_i(\x) = (\x \setminus \{x_i\}) \sqcup \{x'_i\}$ and:
\begin{equation}
\label{eq: exchange relations}
x_ix'_i = \prod_{\substack{1\leqslant j\leqslant n~; \\ b_{ji}>0}} {x_j}^{b_{ji}} + \prod_{\substack{1\leqslant j\leqslant n~; \\ b_{ji}<0}} {x_j}^{-b_{ji}}
\end{equation}
We call (\ref{eq: exchange relations}) the {\it exchange relation} of the mutation, and call \begin{equation}
p_{i}^{+}={\prod_{\substack{1\leqslant j\leqslant n~; \\ b_{ji}>0}} {x_j}^{b_{ji}}}~~~~~~~~~and~~~~~~~~p_{i}^{-}={\prod_{\substack{1\leqslant j\leqslant n~; \\ b_{ji}<0}} {x_j}^{-b_{ji}}}
\end{equation}
the {\it coefficients} of the exchange relation.
\end{defn}
It is easy to check that a seed mutation is an involution, that is $\mu_i\mu_i(\S)=\S$.

\begin{defn}\cite{FZ02}\label{def: cluster algebra}
Denote by ${\X}$ the union of all possible clusters obtained from $\x$ by iterated mutations. The {\it cluster algebra} $\A_Q$ is a $\Z-$subalgebra of the rational function field $\F=\Q(x_1,\cdot \cdot \cdot, x_{n+m})$ generated by cluster variables in ${\X}$.
\end{defn}
We call the variables in $\fx=\{x_{n+1},\dots, x_{n+m}\}$ the {\it frozen
cluster variables} of $\A_Q$, and call the rest variables in ${\X}$ the {\it exchangeable cluster variables} of $\A_Q$. Then in a seed $(\tilde{Q},\tilde{\x})$, $\tilde{\x}=\tilde{\ex}\sqcup\fx$, where $\tilde{\ex}$ is the set of exchangeable cluster variables. Let $\mathbb{P}$ be the {\it coefficient group} of $\A_Q$, which is a multiplicative free abelian group generated by elements in $\fx$. We call the elements in $\mathbb{P}$ the {\it coefficients} of $\A_Q$. We denote by $\A_Q[{\mathbb{P}}^{-1}]$ the localization of $\A_Q$ at $\mathbb{P}$, which is called the {\it localized cluster algebra of $\A_Q$}. Note that $\A_Q[{\mathbb{P}}^{-1}]$ is the initial definition of the geometric cluster algebras given in \cite{FZ02}. A sequence $(y_1, \cdots , y_l)$ is called {\it $(Q,\x)$-admissible} if $y_1$ is exchangeable in $\x$ and $y_i$ is exchangeable in
$\mu_{y_{i-1}} \circ \cdots \circ \mu_{y_1}(\x)$  for every $2 \leqslant i \leqslant l$. A {\it rooted cluster algebra} associated to $\A_Q$ is a triple $(Q,\x,\A_Q)$ \cite{ADS13}. We call the pair $(Q, \x)$ the {\it initial seed} and $\x$ the {\it initial cluster}. We strictly distinguish the cluster algebras and the rooted cluster algebras in the rest of the paper.
It has been shown in \cite{GSV08} that a cluster $\tilde{\x}$ determines the quiver $\tilde{Q}$, and we denote its quiver by $Q(\tilde\x)$ and write $p_{x}$ to the vertex of $\tilde{Q}$ labeled by a variable $x$ in $\tilde{\x}$.\\

\subsection{Cluster automorphisms}
Let $(Q,\x,\A_Q)$ and $(Q',\x',\A_{Q'})$ be two rooted cluster algebras, where $\x=\ex\sqcup\fx$ and $\x'=\ex'\sqcup\fx'$. Recall from \cite{ADS13} that a {\it rooted cluster morphism} $f$ from $(Q,\x,\A_Q)$ to
$(Q',\x',\A_{Q'})$ is a ring homomorphism $f: \A_Q \to \A_{Q'}$ satisfies the following three conditions,\\
({\bf{CM1}})~ $f(\ex) \subset \ex' \sqcup \Z$~;\\
({\bf{CM2}})~ $f(\fx) \subset \x' \sqcup \Z$~;\\
({\bf{CM3}})~ for every $(f,\x,\x')$-biadmissible sequence $(y_1, \cdots , y_l)$ and any $y$ in $\x$, we have $f(\mu_{y_l} \circ \cdots \circ \mu_{y_1,\x}(y)) =
\mu_{f(y_l)} \circ \cdots \circ \mu_{f(y_1),\x'}(f(y))$. Here a {\it $(f,\x,\x')$-biadmissible sequence} $(y_1, \cdots ,
y_l)$ is a $(Q,\x)$-admissible sequence such that $(f(y_1),\cdots ,f(y_l))$ is $(Q',\x')$-admissible.

\begin{defn}
Let $\A$ and $\A'$ be two cluster algebras.
\begin{itemize}
\item[(a)] We call a map $f: \A' \to \A$ a {\it cluster isomorphism} if there exist two seeds $(Q',\x')$ and $(Q,\x)$ of $\A'$ and $\A$
    respectively, and a bijective rooted cluster morphism $\tilde{f}$ from $(Q',\x',\A')$ to $(Q,\x,\A)$ such that $f=\tilde{f}$ on $\A'$.
\item[(b)] We call $f: \A \to \A$ a {\it cluster automorphism} if $f$ is a cluster isomorphism.
\item[(c)] We call $Aut(\A)$ the cluster automorphism group of $\A$, which is a group consisting of cluster automorphisms of $\A$ under
    compositions of automorphisms.
\end{itemize}
\end{defn}
The cluster automorphism of a coefficient free cluster algebra and its equivalent characterizations is given in \cite{ASS12}. For cluster algebras with coefficients, we also have the following characterizations.
\begin{prop}\label{Prop of CA}
A $\Z$-algebra automorphism $f:\A\to\A$ is a cluster automorphism if and only if one of the following conditions is satisfied:
\begin{enumerate}
\item there exists a seed $\S=(\ex,\fx,Q)$ of $\A$, such that $f(\x)$ is the cluster in a seed $\S'=(\ex',\fx',Q')$ of $\A$ with $Q'\cong Q$ or $Q'\cong Q^{op}$, under the correspondence $p_{x}\mapsto p_{f(x)}$;
\item for every seed $\S=(\ex,\fx,Q)$ of $\A$, $f(\x)$ is the cluster in a seed $\S'=(\ex',\fx',Q')$ with $Q'\cong Q$ or $Q'\cong Q^{op}$, under the correspondence $p_{x}\mapsto p_{f(x)}$.
\end{enumerate}
\end{prop}
\begin{proof}
Since $Q$ and its principal part are both connected, the proofs are similar to the proofs of Lemma 2.3 and Proposition 2.4 in \cite{ASS12}.
\end{proof}

\begin{lem}\label{index of direct auto}
We call a cluster automorphism of $\A_Q$ which maps $Q$ to a quiver isomorphic to $Q$ a direct cluster automorphism, then all the direct cluster automorphisms consist a normal subgroup $Aut^+(\A_Q)$ of $Aut(\A)$ with the index at most two.
\end{lem}
\begin{proof}
By using above proposition, the proof is similar to the proof of Lemma 2.9 in \cite{ASS12}.
\end{proof}

\subsection{Coefficient specializations}\label{section coefficient specializations}
The cluster algebra with universal coefficients is defined in \cite{FZ07}, we state its definition in our settings, that is, the universal
geometric cluster algebra.
\begin{defn}\cite{FZ07,R14a}\label{definition coefficient specilazation}
Let $Q$ be a quiver with $n$ exchangeable vertices and without frozen vertices. We denote by $RCA(Q)$ the set of rooted cluster algebras $(\tilde{Q},\tilde{\x}, \A_{\tilde{Q}})$, where $\A_{\tilde{Q}}$ are cluster algebras of
 the ice quivers with the principal part $Q$.
\begin{itemize}
\item[(a)] Let $(Q',\x',\A_{Q'})$ and $(Q'',\x'',\A_{Q''})$ be two rooted cluster algebras in $RCA(Q)$. Let $\mathbb{P}'$ and $\mathbb{P}''$
    be the coefficient groups of $\A_{Q'}$ and $\A_{Q''}$ respectively. If there is a group homomorphism $\varphi: \mathbb{P}'\to
    \mathbb{P}''$ that extends to a ring homomorphism $\varphi: \A_{Q'}[{\mathbb{P}'}^{-1}] \to \A_{Q''}[{\mathbb{P}''}^{-1}]$ such that
    $\varphi(\mu_\varpi(x'_i))=\mu_\varpi(x''_i)$ for any $1 \leqslant i \leqslant n$ and any admissible sequence $\varpi$, we say that both
    $\varphi: \mathbb{P}'\to \mathbb{P}''$ and $\varphi: \A_{Q'}[{\mathbb{P}'}^{-1}] \to \A_{Q''}[{\mathbb{P}''}^{-1}]$ are coefficient
    specializations.
\item[(b)] Let $(Q^{univ},\x^{univ},\A_{Q^{univ}})$ be a rooted cluster algebra in $RCA(Q)$ with coefficient group $\mathbb{P}$. We call
    $\A_{Q^{univ}}$ a universal geometric cluster algebra of $RCA(Q)$, if for any rooted cluster algebra $(Q',\x',\A_{Q'})$ in $RCA(Q)$
    with coefficient group $\mathbb{P'}$, there is a unique coefficient specialization from $\A_{Q^{univ}}[{\mathbb{P}}^{-1}]$ to
    $\A_{Q'}[{\mathbb{P}'}^{-1}]$.
\end{itemize}
\end{defn}

\begin{rem}
Since the cluster algebras we considered are geometric ones, the universal geometric cluster algebra is different from the initial cluster algebra with universal coefficients given in \cite{FZ02}. Some deformed universal geometric cluster algebras and their existences have been considered in a series of papers \cite{R12,R14a,R14b}.
\end{rem}
The following is an equivalent characterization of coefficient specializations.
\begin{lem}\cite{FZ07}\label{EC of specialization}
Let $(Q',\x',\A_{Q'})$ and $(Q'',\x'',\A_{Q''})$ be two rooted cluster algebras in $RCA(Q)$. Let $\mathbb{P}'$ and $\mathbb{P}''$ be the
coefficient groups of $\A_{Q'}$ and $\A_{Q''}$ respectively. A group homomorphism $\varphi: \mathbb{P}'\to \mathbb{P}''$ is a coefficient
specialization if and only if
$$\varphi(p'^{\pm}_{i,\varpi})=p''^{\pm}_{i,\varpi}$$
for any exchangeable vertex $i$ and admissible sequence $\varpi$, where $p'^{\pm}_{i,\varpi}$ and $p''^{\pm}_{i,\varpi}$ are coefficients in
exchange relations of $(\mu_{\varpi}(Q'),\mu_{\varpi}(\x'))$ and of $(\mu_{\varpi}(Q''),\mu_{\varpi}(\x''))$ respectively.
\end{lem}

It follows from \cite{FZ07} that a cluster algebra with universal coefficients (if there exists one) is unique. In our settings, we have the following more specific characterization for the uniqueness of the universal geometric cluster algebras.

\begin{lem}\label{uniqueness of universal algebra}
Let $(Q',\x',\A_{Q'})$ and $(Q'',\x'',\A_{Q''})$ be two rooted cluster algebras in $RCA(Q)$ with $\A_{Q'}$ and $\A_{Q''}$ the universal
geometric cluster algebras. We denote by $B=(b_{ji})_{n\times n}$ the exchange matrix corresponding to $Q$, and by ${B \choose B'}=(b'_{ji})_{(n+m')\times n}$ (or ${B \choose B''}=(b''_{ji})_{(n+m'')\times n}$ respectively) the
exchange matrix corresponding to $Q'$ (or $Q''$ respectively).
Then $m'=m''$ and there is an invertible matrix $A=(a_{jk})_{m'\times m''}$ with entries in $\mathbb{Z}$ such that

$$A^{t}\mu_{\varpi}(B')=\mu_{\varpi}(B'')$$

for any admissible sequence $\varpi$, where $\mu_{\varpi}(B')$ and $\mu_{\varpi}(B'')$ are frozen parts of $\mu_{\varpi}{B \choose B'}$ and
$\mu_{\varpi}{B \choose B''}$ respectively.
\end{lem}

\begin{proof}
Let $\mathbb{P}'$ and $\mathbb{P}''$ be the coefficient groups of $\A_{Q'}$ and $\A_{Q''}$ respectively. Since $\A_{Q'}$ and $\A_{Q''}$ are both universal geometric cluster algebras with the same principal part, there are two coefficient specializations $\varphi: \mathbb{P}' \to \mathbb{P}''$ and $\psi: \mathbb{P}'' \to \mathbb{P}'$. Then the composition $\psi\varphi$ is a coefficient specialization from $\mathbb{P}'$ to itself and thus $\psi\varphi$ is the identity. Since $\mathbb{P}'$ and $\mathbb{P}''$ are both free abelian groups, they have the same rank, thus $m'=m''$. We assume that
$$\varphi({x'_{n+j}})={\displaystyle\prod_{1 \leqslant k \leqslant m''}}{x''_{n+k}}^{a_{jk}}$$
for any $1 \leqslant j \leqslant m'$, where each $a_{jk}$ is an integer number. Then we have a matrix $A=(a_{jk})$ which is invertible since
$\varphi$ is an isomorphism.
For any $1 \leqslant i \leqslant n$, we consider the exchange relations
$$\mu_{i}({x'_i}) x'_i = {\prod_{b'_{ji} > 0}} {x'_{j}}^{b'_{ji}} + \prod_{b'_{ji}< 0} {x'_{j}}^{-b'_{ji}}$$
and
$$\mu_{i}({x''_i}) x''_i = \prod_{b''_{ki} > 0} {x''_{k}}^{b''_{ki}} + \prod_{b''_{ki}< 0} {x''_{k}}^{-b''_{ki}}.$$
Since $\varphi({x'_{j}})=x''_{j}$ for any $1 \leqslant j \leqslant n$, the equality $\varphi(\mu_{i}({x'_{i}})=\mu_{i}(x''_{i})$ is equivalent
to
$${\prod_{\substack{1 \leqslant j \leqslant m' \\ b'_{ji}>0}}(\prod_{1 \leqslant k \leqslant
m''}{x''_{n+k}}^{a_{jk}})^{b'_{(n+j)i}}=\prod_{\substack{1 \leqslant j \leqslant m'' \\ b''_{ji}>0}}{x''_{n+k}}^{b''_{(n+k)i}}}~~~~~~~~~~and$$

$${\prod_{\substack{1 \leqslant j \leqslant m' \\ b'_{ji}<0}}(\prod_{1 \leqslant k \leqslant
m''}{x''_{n+k}}^{a_{jk}})^{-b'_{(n+j)i}}=\prod_{\substack{1 \leqslant j \leqslant m'' \\ b''_{ji}<0}}{x''_{n+k}}^{-b''_{(n+k)i}}}.$$

Thus we have
$${\prod_{1 \leqslant j \leqslant m'}(\prod_{1 \leqslant k \leqslant m''}{x''_{n+k}}^{a_{jk}})^{b'_{(n+j)i}}=\prod_{1 \leqslant k \leqslant
m''}{x''_{n+k}}^{b''_{(n+k)i}}},$$
and finally
$${{\sum_{1 \leqslant k \leqslant m''}}{a_{jk}b'_{(n+j)i}=b''_{(n+k)i}}}.$$
Therefore we have $A^{t}B'=B''.$ By induction, we can prove the general case $A^{t}\mu_{\varpi}(B')=\mu_{\varpi}(B'')$ for any admissible
sequence $\varpi$.
\end{proof}
The existence of universal geometric cluster algebras is not clearly at all. A cluster algebra of finite type with universal coefficients is given in $\cite{FZ02}$. It is a geometric cluster algebra, and its frozen rows are given by $g$-vectors of the corresponding cluster algebra with principal coefficients, we call it $FZ$-universal cluster algebra.
\begin{exm}\label{two univ algs}
We consider the $FZ$-universal cluster algebra $\A_{Q'}$ of type $A_2$, whose quiver $Q'$ is as follows with matrix ${B \choose B'}$, where $1$ and $2$ are exchangeable vertices and the others are frozen ones.

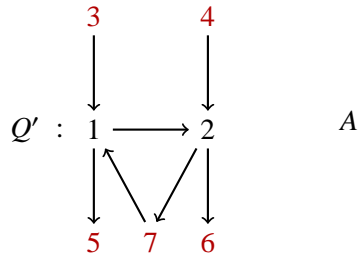
\begin{center}
\begin{tikzpicture}
\node[] (C) at (-5,0) 						{\begin{tikzpicture}[scale=0.5]
                        \node[] () at (-3,0) {$Q'~:$};
                        \node[] () at (-1.5,0) {$1$};
                        \node[] () at (1.5,0) {$2$};
                        \node[] () at (-1.5,3) {$\textcolor[rgb]{0.70,0.00,0.00}{3}$};
                        \node[] () at (1.5,3) {$\textcolor[rgb]{0.70,0.00,0.00}{4}$};
                        \node[] () at (-1.5,-3) {$\textcolor[rgb]{0.70,0.00,0.00}{5}$};
                        \node[] () at (1.5,-3) {$\textcolor[rgb]{0.70,0.00,0.00}{6}$};
                        \node[] () at (0,-3) {$\textcolor[rgb]{0.70,0.00,0.00}{7}$};
						\draw[->,thick] (-1,0) -- (1,0);
						\draw[->,thick] (-1.5,2.5) -- (-1.5,0.5);
						\draw[->,thick] (1.5,2.5) -- (1.5,0.5);
						\draw[<-,thick] (-1.5,-2.5) -- (-1.5,-0.5);
						\draw[<-,thick] (1.5,-2.5) -- (1.5,-0.5);
						\draw[<-,thick] (0.15,-2.45) -- (1.2,-0.5);
						\draw[->,thick] (-0.15,-2.45) -- (-1.2,-0.5);
						\end{tikzpicture}};
\node[] (C) at (-2,0)
                       {\begin{tikzpicture}
                        $\left(
                           \begin{array}{c}
                             B \\
                             B' \\
                           \end{array}
                         \right)
$
$~=~$
                        $\left(
                          \begin{array}{cc}
                            0 & 1 \\
                            -1 & 0 \\
                            1 & 0 \\
                            0 & 1 \\
                            -1 & 0 \\
                            0 & -1 \\
                            1 & -1 \\
                          \end{array}
                        \right)$;
                        \end{tikzpicture}};
\node[] (C) at (2,0)  {$~~$};

\end{tikzpicture}
\end{center}

Given the following invertible matrix $A$, denote by $Q''$ the ice quiver of ${B \choose B''}$, where $B''=A^{t}B'$.
\begin{center}
\begin{tikzpicture}
\node[] (C) at (-6,0)
                       {\begin{tikzpicture}
                        \node[] () at (-2,0) {$A~:$};
                        \node[] () at (0,0)
                               {$\left(
                                  \begin{array}{ccccc}
                                    2 & 1 & 0 & 0 & 0 \\
                                    1 & 1 & 0 & 0 & 0 \\
                                    0 & 0 & 3 & 2 & 0 \\
                                    0 & 0 & 1 & 1 & 0 \\
                                    0 & 0 & 0 & 0 & 1 \\
                                  \end{array}
                                \right)$};
                        \end{tikzpicture}};
\node[] (C) at (-2,0) 						{\begin{tikzpicture}[scale=0.5]
                        \node[] () at (-3,0) {$Q''~:$};
                        \node[] () at (-1.5,0) {$1$};
                        \node[] () at (1.5,0) {$2$};
                        \node[] () at (-1.5,3) {$\textcolor[rgb]{0.70,0.00,0.00}{3}$};
                        \node[] () at (1.5,3) {$\textcolor[rgb]{0.70,0.00,0.00}{4}$};
                        \node[] () at (-1.5,-3) {$\textcolor[rgb]{0.70,0.00,0.00}{5}$};
                        \node[] () at (1.5,-3) {$\textcolor[rgb]{0.70,0.00,0.00}{6}$};
                        \node[] () at (0,-3) {$\textcolor[rgb]{0.70,0.00,0.00}{7}$};
						\draw[->,thick] (-1,0) -- (1,0);
						\draw[->,thick] (-1.6,2.5) -- (-1.6,0.5);
						\draw[->,thick] (-1.4,2.5) -- (-1.4,0.5);
						\draw[->,thick] (-1,2.5) -- (1,0.5);
						\draw[->,thick] (1,2.5) -- (-1,0.5);
						\draw[->,thick] (1.5,2.5) -- (1.5,0.5);
						\draw[<-,thick] (-1.5,-2.5) -- (-1.5,-0.5);
						\draw[<-,thick] (-1.7,-2.5) -- (-1.7,-0.5);
						\draw[<-,thick] (-1.3,-2.5) -- (-1.3,-0.5);
						\draw[->,thick] (1,-0.5) -- (-1,-2.5);
						\draw[<-,thick] (-1,-2.5) -- (1,-0.5);
						\draw[<-,thick] (1.1,-2.5) -- (-1,-0.3);
						\draw[<-,thick] (1,-2.6) -- (-1.1,-0.4);
						\draw[<-,thick] (1.5,-2.5) -- (1.5,-0.5);
						\draw[<-,thick] (0.15,-2.45) -- (1.2,-0.5);
						\draw[->,thick] (-0.15,-2.45) -- (-1.2,-0.5);
						\end{tikzpicture}};
\node[] (C) at (0.5,0)
                       {\begin{tikzpicture}
                        $\left(
                           \begin{array}{c}
                             B \\
                             B'' \\
                           \end{array}
                         \right)
$
$~=~$
                        $\left(
                          \begin{array}{cc}
                            0 & 1 \\
                            -1 & 0 \\
                            2 & 1 \\
                            1 & 1 \\
                            -3 & -1 \\
                            -2 & -1 \\
                            1 & -1 \\
                          \end{array}
                        \right)$;
                        \end{tikzpicture}};
\node[] (C) at (4,0)  {$~~$};
\end{tikzpicture}
\end{center}
Then as the construction in above lemma, the matrix $A$ induces a group homomorphism $\varphi$ from the coefficient group $\mathbb{P}'$ of $\A_{Q'}$ to the coefficient group $\mathbb{P}''$ of $\A_{Q''}$, and it is straightforward to check that $\varphi$ is a coefficient specialization. Thus $\A_{Q''}$ is another universal cluster algebra of type $A_2$, which is not a $FZ$-universal cluster algebra. We will show in Example \ref{universal subgroup} that the cluster automorphism group $Aut(\A_{Q'})$ is $D_5$ which is the biggest one in the cluster automorphism groups of gluing free cluster algebras of type $A_2$, while $Aut(\A_{Q''})$ is $\{id\}$ which is the smallest one.
\end{exm}

\subsection{Cluster algebras from surfaces}
Following \cite{FST08}, we recall oriented marked Riemann surfaces and the cluster algebras arising from surfaces.
An {\it oriented marked Riemann surface} (or surface for brevity) is a pair $(S,M)$, where $S$ is a connected oriented Riemann surface with boundary and $M$ is a finite set of {\it marked points} on $S$. The marked points in the interior of $S$ are called {\it punctures}. We assume that $S$ is not closed and there is at least one marked point on each connected component of the boundary. We always assume that, for some technical reasons, $(S,M)$ is none of a disc with one, two or three marked points on the boundary; or a once-punctured disc with one marked point on the boundary.\\

Two curves in $(S,M)$ are the same if they are isotopic with respect to marked points. Two curves $\gamma$ and $\gamma'$ are {\it compatible}
if there are representatives of $\gamma$ and $\gamma'$ which do not intersect in $S\setminus M$.
An {\it arc} is a non-contracted curve, with endpoints in $M$ which are the only possible self-intersection points. An arc is a {\it boundary arc} if it is isotopic to a boundary component with respect to marked points, otherwise it is an internal arc. An {\it ideal triangulation} of $(S,M)$ is a maximal collection of compatible internal arcs.
An {\it ideal triangulation of $(S,M)$ with boundary}\cite{ST09,ADS13} is a maximal collection of compatible arcs (which contains the boundary arcs). The arcs
of the triangulation cut the surface into triangles (which may be self-folded). Given an ideal triangulation $T$ without self-folded triangles (see Section 4 in \cite{FST08} for the general case), we associate to it a quiver $Q_T$ as follows.
\begin{itemize}
\item[(a)] the vertices in $Q_T$ are the arcs in $T$;
\item[(b)] if two arcs $\gamma$ and $\gamma'$ are both sides of a triangle and $\gamma'$ follows $\gamma$ in the positive direction, then add
    an arrow from $\gamma$ to $\gamma'$;
\item[(c)] deleting all the $2$-cycles.
\end{itemize}
For the corresponding ideal triangulation $T^{b}$ with boundary, we associate to it an ice quiver $Q_{T^{b}}$ in a similar way (we delete the arrows between boundary arcs), where the frozen vertices are the boundary arcs. Then $Q_T$ is the principal part of $Q_{T^{b}}$. We associate cluster algebras $\A_{Q_{T}}$ and
$\A_{Q_{T^{b}}}$ to $Q_T$ and $Q_{T^{b}}$ respectively. It is proved in \cite{FST08} that the cluster algebras $\A_{Q_{T}}$ and
$\A_{Q_{T^{b}}}$ are independent on the choice of the triangulations. We call them the cluster algebras of $(S,M)$, without boundary and with boundary respectively, and denote them by $\A(S,M)$ and $\A^{b}(S,M)$ respectively. We obtain a new triangulation by flipping an ideal triangulation
without self-folded triangles, where the {\it flip} of the triangulation corresponds to the mutation of the quiver. To flip an ideal
triangulation at the internal arc of a self-folded triangle, Fomin, Shapiro and Thurston \cite{FST08} introduced the tagged triangulation by adding tags $0$
or $1$ on the punctures. These triangulations one-to-one correspond to the clusters of $\A_{Q_{T}}$ and
induce one-to-one correspondences from the tagged arcs to the cluster variables. It is showed in \cite{FST08} that any ideal triangulation,
except for a triangulation of a 4-punctured sphere (which is closed and thus eliminated in our assumptions of the surfaces), can be obtained by gluing the
{\it puzzle pieces} in Figure \ref{puzzle-pieces} along the matched sides (see Section 4 of \cite{FST08} for details). The following lemma is useful
in the paper.
\begin{figure}[htbp]
\begin{center}
\begin{tabular}{ccc}
\setlength{\unitlength}{2pt}
\begin{picture}(20,19)(10,-3)
{\begin{tikzpicture}
						\draw[-,thick] (0,0) -- (1.7,0);
						\draw[-,thick] (0,0) -- (0.9,1.2);
						\draw[-,thick] (0.9,1.2) -- (1.7,0);
                        \node[] () at (0,0) {$\bullet$};
                        \node[] () at (1.7,0) {$\bullet$};
                        \node[] () at (0.9,1.2) {$\bullet$};
						\end{tikzpicture}}
\end{picture}
&
\setlength{\unitlength}{8pt}
\begin{picture}(10,10)(0,0)
\thicklines
\qbezier(5,0)(0,5)(5,10)
\qbezier(5,0)(10,5)(5,10)
\qbezier(5,0)(2,7)(5,7)
\qbezier(5,0)(8,7)(5,7)
\put(5,0){\line(0,1){5}}
\put(4.67,4.8){{$\bullet$}}
\put(4.67,9.5){{$\bullet$}}
\put(4.67,-0.2){{$\bullet$}}
\end{picture}
&
\setlength{\unitlength}{1pt}
\begin{picture}(80,80)(-25,0)
\thicklines
\qbezier(20,0)(-30,45)(0,45)
\qbezier(20,0)(70,45)(40,45)
\qbezier(20,0)(20,45)(0,45)
\qbezier(20,0)(20,45)(40,45)
\qbezier(20,0)(4,28)(4,28)
\qbezier(20,0)(36,28)(36,28)
\qbezier(20,0)(-67,60)(20,60)
\qbezier(20,0)(107,60)(20,60)

\put(36,28){\makebox(0,0){$\bullet$}}
\put(4,28){\makebox(0,0){$\bullet$}}

\end{picture}
\\[.1in]
\end{tabular}
\end{center}
\caption{The three puzzle pieces}
\label{puzzle-pieces}
\end{figure}
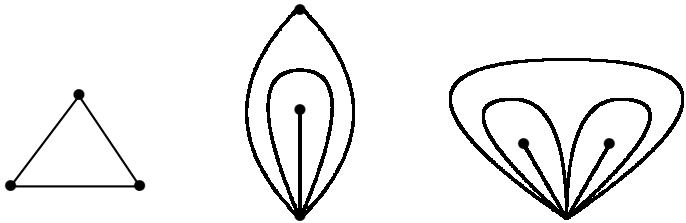

\begin{lem}\cite{FST08}(Lemma 2.13)\label{lem in Surface}
The surface $(S,M)$ has an ideal triangulation which contains no self-folded triangles.
\end{lem}

For a surface $(S,M)$, the marked mapping class group describes the symmetries of the surface $S$ with respect to the marked points and the tags
of the punctures. We recall its definition from \cite{ASS12}. We denote by $\homeo$ the group of orientation-preserving homeomorphisms from $S$
to $S$ that fix $M$ setwise, and $\homeoo$ the subgroup of $\homeo$ which consisting of homeomorphisms that are isotopic to the identity
relative to $M$. Then the {\it mapping class group} $\mg$ is defined as the quotient $$\mg=\homeo/\homeoo.$$ Note that the group $\mg$ acts on
the set of ideal triangulations in the natrual way.
Let $\mathcal{Z}$ be the power set of the set of punctures of $(S,M)$. A homeomorphism in $\homeo$ acts on $\mathcal{Z}$ pointwise. The
set $\mathcal{Z}$ is a group with respect to the operation $\ominus: \mathcal {P}_1\ominus \mathcal {P}_2= (\mathcal {P}_1\cup \mathcal
{P}_2)\setminus (\mathcal {P}_1\cap \mathcal {P}_2)$. We define the {\it marked mapping class group} $\mmg$ of the surface $(S,M)$ to be the
semidirect product
$$\mmg=\mathcal{Z}\rtimes\mg,$$ where the group operation is given by
$(\bar f_1,\mathcal{P}_1)(\bar f_2, \mathcal {P}_2) =(\bar f_1\bar f_2, \mathcal {P}_1\ominus f_1(\mathcal {P}_2))$ for any $\mathcal{P}_1,
\mathcal{P}_2$ in $\mathcal{Z}$ and any $f_1, f_2$ in $\homeo$. This group acts on the set of tagged triangulations of $(S,M)$, with the
$\mathcal{Z}$ part acting by simultaneously changing the tags of the arcs attached to the punctures.
On the one hand, since the orientation-preserving homeomorphism of $(S,M)$ and the changing of tags maintain the combinatorics of the tagged triangulations, an
element in $\mmg$ induces a direct cluster automorphism of the associating cluster algebra $\A(S,M)$ (see section 4.4 in \cite{ASS12} for more
details). On the other hand, a direct cluster automorphism in $Aut^+(\A(S,M))$ induces an isomorphism between the quivers of two tagged
triangulations of $(S,M)$, and it is proved in \cite{BS13} (Subsection 8.7) that except some special cases, the two tagged
triangulations differ by an element in the marked mapping class group $\mmg$. Thus the marked mapping
class group $\mmg$ and the direct automorphism group $Aut^+(\A(S,M))$ are isomorphic. More precisely, we have the following result:\\

Let $(S,M)$ be a surface satisfies the following\\

{\bf{Assumption 1:}} Suppose that $(S,M)$ is not one of the following surfaces
\begin{itemize}
\item[(a)] a once-punctured disc with two or four marked points on the boundary;
\item[(b)] a twice-punctured disc with two marked points on the boundary;
\item[(c)] a closed surface (which we have eliminated in the initial assumption of the surfaces).
\end{itemize}
Then the marked mapping class group $\mmg$ is isomorphic to the direct automorphism group $Aut^+(\A(S,M))$ (see also in \cite{BQ12}).

\section{Cluster automorphism groups and coefficients}\label{section main}

We study in this section the influence of coefficients on cluster automorphism groups.

\subsection{Gluing free cluster algebras}\label{section gluing free cluster algebras}

In this subsection, we fix $(Q,\x,\A_Q)$ a rooted cluster algebra with exchange matrix $B=(b_{ji})_{(m+n)\times n}$.
\begin{defn}\label{definition gluing free cluster algebras}
We say that
\begin{itemize}
\item[(a)] two frozen vertices $j$ and $k$ are glueable in $Q$, if there is a positive rational number $n_{jk}$ such that
    $b_{ji}=n_{jk}b_{ki}$ for any exchangeable vertex $i$ in $Q$;
\item[(b)] two frozen vertices $j$ and $k$ are strictly glueable in $Q$, if they are glueable with $n_{jk}=1$;
\item[(c)] a frozen vertex $j$ is prime, if $g.c.d~(b_{j1}, \cdots, b_{jn})=1$;
\item[(d)] $Q$ is strictly gluing free, if any two frozen vertices in $Q$ are not glueable;
\item[(e)] $Q$ is gluing free, if any two frozen vertices in $Q$ are not strictly glueable;
\item[(f)] $Q$ is prime, if any frozen vertex is prime;
\item[(g)] $Q$ is prime gluing free, if $Q$ is prime and strictly gluing free;
\item[(h)] the rooted cluster algebra $(Q,\x,\A_Q)$ is gluing free, if $Q$ is gluing free;
\item[(i)] the rooted cluster algebra $(Q,\x,\A_Q)$ is prime gluing free, if $Q$ is prime gluing free.
\end{itemize}
\end{defn}
\begin{rem}
Note that $Q$ is prime gluing free if and only if it is prime and gluing free.
\end{rem}
We separate the frozen vertices of $Q$ into sets of glueable vertices, then we have a disjoint collection of sets $\{j_{1,1}, \cdots ,
j_{1,t_1}\}, \cdots , \{j_{s,1}, \cdots , j_{s,t_s}\}$ of frozen vertices of $Q$, where any two vertices (if there exist) in each set are
glueable. Note that there may exist set with a single vertex. In fact $Q$ is strictly gluing free if and only if all of these sets have a
single vertex. The vertex separation determines a separation of coefficients in $\A_Q$. We denote the corresponding sets by
$\{x_{j_{1,1}}, \cdots , x_{j_{1,t_1}}\},  \cdots , \{x_{j_{s,1}}, \cdots , x_{j_{s,t_s}}\}$ respectively. Then we have the following
\begin{thm}\label{main thm 1 of pgq}  For a rooted cluster algebra $(Q,\x,\A_Q)$, there is a unique prime gluing free rooted cluster algebra $(Q_{pgf},\x_{pgf},\A_{Q_{pgf}})$ and a unique injective coefficient specialization
$\varphi: \mathbb{P}_{pgf} \to \mathbb{P}$ such that for any $1 \leqslant k \leqslant s$, there is a coefficient $x_{j_k}$ in
$\mathbb{P}_{pgf}$ which is mapped to $x_{j_{k,1}}^{n_{j_{k,1}}} x^{n_{j_{k,2}}}_{j_{k,2}}\cdots x^{n_{j_{k,t_k}}}_{j_{k,t_k}}$ in $\mathbb{P}$, where
$n_{j_{k,1}}, n_{j_{k,2}}, \cdots , n_{j_{k,t_k}}$ are positive integer numbers, and $\mathbb{P}_{pgf}$ and $\mathbb{P}$ are coefficient groups
of $\A_{Q_{pgf}}$ and $\A_{Q}$ respectively.
\end{thm}
\begin{proof}
We prove this by several steps:\\

Step 1. We construct the needed prime gluing free quiver $Q_{pgf}$ (write as $\overline{Q}$ for simplicity) from $Q$. The principal part of
$\overline{Q}$ is the same as the principal part of $Q$. There are $s$ frozen vertices $\{j_1, \cdots, j_s\}$ in $\overline{Q}$, which is the
set of representatives of the sets of glueable vertices in $Q$.
We define the associating matrix $\overline{B}=(\overline{b}_{ji})_{(n+s)\times n}$ of $\overline{Q}$ as follows.
For each frozen vertex $j_k, 1 \leqslant k \leqslant s$, let $n_{j_{k,l}}$ in $\mathbb{Z}^{+}$, $1 \leqslant l \leqslant
s_k$, be the greatest common divisor of $\{b_{j_{k,l}{i}}, 1 \leqslant i \leqslant n\}$. Then we have $g.c.d~(\frac{b_{j_{k,l}1}}{n_{j_{k,l}}},
\cdots, \frac{b_{j_{k,1}{n}}}{n_{j_{k,l}}})=1$. Let
$$\overline{b}_{j_ki}=\frac{b_{j_{k,l}i}}{n_{j_{k,l}}}$$ for any $1 \leqslant k \leqslant s$ and any $1 \leqslant i \leqslant n$. Note that
$\overline{b}_{j_ki}$ is independent on the choice of $l$. In fact, all of the order sets $\{\frac{b_{j_{k,l}1}}{n_{j_{k,l}}}, \cdots,
\frac{b_{j_{k,1}n}}{n_{j_{k,l}}}\}$, $1 \leqslant l \leqslant s_k$, are the same, since all of vertices $j_{k,l}, 1 \leqslant l \leqslant s_k$
are glueable with each other. Then by the construction, $\overline{Q}$ is prime gluing free.\\

Step 2. To show that $\varphi: \overline{\mathbb{P}} \to \mathbb{P}$ gives a coefficient specialization, by lemma \ref{EC of specialization},
we only need to prove the equalities
\begin{eqnarray}\label{3}
&\varphi(\overline{p}^{\pm}_{i,\varpi})=p^{\pm}_{i,\varpi}
\end{eqnarray}
for any exchangeable vertex $i$ and admissible sequence $\varpi$.\\
For the case $\varpi=0$,
$\overline{p}_{i}^{+}=\prod\limits_{\overline{b}_{j_ki}>0} x_{j_{k}}^{\overline{b}_{j_ki}}$ and $p_{i}^{+}=\prod\limits_{b_{j_{k,l}i}>0}
x_{j_{k,l}}^{b_{j_{k,l}i}}$, we have equalities:
$$\begin{array}{rcl}
				\varphi(\overline{p}_{i}^{+}) & = & \prod\limits_{\overline{b}_{j_ki}>0} \varphi(x_{j_{k}})^{\overline{b}_{j_ki}} \\
			        & = & \prod\limits_{\overline{b}_{j_ki}>0} (x_{j_{k,1}}^{n_{j_{k,1}}}\cdots
x_{j_{k,t_k}}^{n_{j_{k,t_k}}})^{\overline{b}_{j_ki}} \\
				    & = & \prod\limits_{\overline{b}_{j_ki}>0} x_{j_{k,1}}^{b_{j_{k,1}i}}\cdots x_{j_{k,t_k}}^{b_{j_{k,t_k}i}} \\
				    & = & \prod\limits_{b_{j_{k,l}i}>0} x_{j_{k,l}}^{b_{j_{k,l}i}}\\
				    & = & p_{i}^{+}.
			\end{array}$$
Similarly, we have $\varphi(\overline{p}_{i}^{-})=p_{i}^{-}$. \\
For the case $\varpi=(i')$, where $i'$ is an exchangeable vertex, we denote by $(Q',\x',\A_{Q'})$ and
$(\overline{Q}',\overline{\x}',\A_{\overline{Q}'})$ the mutations $(\mu_{i'}(Q),\mu_{i'}(\x),\A_{\mu_{i'}(Q)})$ and
$(\mu_{i'}(\overline{Q}),\mu_{i'}(\overline{\x}),\A_{\mu_{i'}(\overline{Q})})$ respectively. We denote by $B'=(b'_{ji})_{(n+m)\times n}$ and
$\overline{B}'=(\overline{b}'_{ji})_{(n+s)\times n}$ the corresponding exchange matrices of $Q'$ and $\overline{Q}'$ respectively. Then it is
clearly that the exchange part of $Q'$ and $\overline{Q}'$ are the same. For any frozen vertex $j_{k,l}$ in $Q$,
$$b'_{j_{k,l}i}= \left\{\begin{array}{ll}
				b_{j_{k,l}i} + b_{j_{k,l}i'}b_{i'i} &if~ i\neq i'~ and~ b_{j_{k,l}i'}b_{i'i}\geqslant 0;\\
				b_{j_{k,l}i} - b_{j_{k,l}i'}b_{i'i} &if~ i\neq i'~ and~  b_{j_{k,l}i'}b_{i'i}< 0;\\
				-b_{j_{k,l}i} &if~ i=i'.
			\end{array}\right.$$
For any frozen vertex $j_{k}$ in $\overline{Q}$,
$$\overline{b}'_{j_ki}= \left\{\begin{array}{ll}
				\overline{b}_{j_ki} + \overline{b}_{j_ki'}\overline{b}_{i'i} &if~ i\neq i'~ and~ \overline{b}_{j_ki}\overline{b}_{i'i}\geqslant
0;\\
\overline{b}_{j_ki} - \overline{b}_{j_ki'}\overline{b}_{i'i} &if~ i\neq i'~ and~ \overline{b}_{j_ki}\overline{b}_{i'i}< 0;\\
				-\overline{b}_{j_ki} &if~ i=i'.
			\end{array}\right.$$
Note that $n_{j_{k,l}}>0$, then from $b_{j_{k,l}i}=n_{j_{k,l}}\overline{b}_{j_ki}$, $b_{j_{k,l}i'}=n_{j_{k,l}}\overline{b}_{j_ki'}$ and
$b_{i'i}=\overline{b}_{i'i}$, we have $b'_{j_{k,l}i}=n_{j_{k,l}}\overline{b}'_{j_ki}$. Let $g=g.c.d~(\overline{b}'_{j_k1}, \cdots,
\overline{b}'_{j_kn})$, then $g ~|~ \overline{b}'_{j_ki'}=-\overline{b}_{j_ki'}$, and thus $g ~|~ \overline{b}_{j_ki}$ for any $1 \leqslant i
\leqslant n$. Because $g.c.d~(\overline{b}_{j_k1}, \cdots, \overline{b}_{j_kn})=1$, we have $g=1$ .
Thus $g.c.d~(b'_{j_{k,l}{1}}, \cdots, b'_{j_{k,l}{n}})=n_{j_{k,l}}$.
Therefore from the construction of the prime gluing free quiver, we have $(\mu_{i'}(Q))_{pgf}=\mu_{i'}(Q_{pgf})$. Then by the case $\varpi=0$,
$\varphi(\overline{p}^{\pm}_{i,(i')})=p^{\pm}_{i,(i')}$. Finally, we prove the equalities \ref{3} by inductions.\\

Step 3. Since $\mathbb{P}_{pgf}$ and $\mathbb{P}$ are both free abelian groups, from the construction of the coefficient specialization
$\varphi$, it is not hard to see that $\varphi$ is injective. The injectivity of $\varphi$ yields that there are no more coefficients in $\A_{Q_{pgf}}$, excepting the ones $j_1, \cdots, j_s$, this guarantees the uniqueness of $(Q_{pgf},\x_{pgf},\A_{Q_{pgf}})$.
\end{proof}

\begin{rem}\label{rem of gf CA}
\begin{itemize}
\item[(a)] It is not hard to see that $\varphi$ not only maps $\A_{Q_{pgf}}[\mathbb{P}_{pgf}^{-1}]$ to $\A_Q[{\mathbb{P}}^{-1}]$, but also
    induces an injective algebra homomorphism from $\A_{Q_{pgf}}$ to $\A_Q$, which maps exchangeable cluster variables to exchangeable ones.
    We call $Q_{pgf}$ the prime gluing free quiver of $Q$ and $(Q_{pgf},\x_{pgf},\A_{Q_{pgf}})$ the prime gluing free rooted cluster algebra
    of $(Q,\x,\A_Q)$.
\item[(b)] We consider the disjoint collection of sets $\{j_{1,1}, \cdots , j_{1,t_1}\}, \cdots , \{j_{s,1}, \cdots , j_{s,t_s}\}$ of frozen
    vertices of $Q$, where any two vertices (if there exist) in each set are strictly glueable. Then we define $Q_{gf}$ as the quiver obtained
    from $Q$ by deleting the vertices $j_{k,l}$ for any $1 \leqslant k \leqslant s$ and $2 \leqslant l \leqslant s_k$. We denote by
    $j_k=j_{k,1}$ for simplicity. Then $Q_{gf}$ is gluing free. We call it the gluing free quiver of $Q$. It is not hard to see that we have
    similar results as in above theorem. That is, there is a unique injective coefficient specialization $\varphi: \mathbb{P}_{gf} \to
    \mathbb{P}$ such that for any $1 \leqslant k \leqslant s$, $\varphi(x_{j_k})=x_{j_{k,1}} x_{j_{k,2}}\cdots x_{j_{k,t_k}}$. Similarly,
    $\varphi$ induces an injective algebra homomorphism from $\A_{Q_{gf}}$ to $\A_Q$, which maps exchangeable cluster variables to
    exchangeable ones. We call $(Q_{gf},\x_{gf},\A_{Q_{gf}})$ the gluing free rooted cluster algebra of $(Q,\x,\A_Q)$.
\end{itemize}
\end{rem}

\begin{defn}
Let $(Q',\x',\A_{Q'})$ and $(Q'',\x'',\A_{Q''})$ be two rooted cluster algebras.
We say that $Q'$ ($(Q',\x',\A_{Q'})$ respectively) and $Q''$ ($(Q'',\x'',\A_{Q''})$ respectively) are gluing
equivalent, if their prime gluing free quivers are isomorphic.
\end{defn}

Note that the gluing equivalent relation gives rise to an equivalent relation on ice quivers (thus on rooted cluster algebras). By this
equivalent relation, we separate ice quivers (rooted cluster algebras respectively) to gluing equivalent classes, which consist of the gluing
equivalent ice quivers (gluing equivalent rooted cluster algebras respectively). Each equivalent class is uniquely determined by a prime gluing
free ice quiver. This prime gluing free ice quiver is a universal element in the sense of that there is a coefficient specialization from this quiver to any quiver in the equivalent class.\\

\begin{lem}\label{main lem in CS}
Let $Q$ be an ice quiver and $\overline{Q}$ be its prime gluing free quiver (or gluing free quiver).
\begin{itemize}
\item[(a)] For any admissible sequence $\varpi$, $\mu_\varpi(\overline{Q})$ is the prime gluing free quiver (or gluing free quiver
    respectively) of $\mu_\varpi(Q)$;
\item[(b)] For any admissible sequence $\varpi$, $\mu_\varpi(\overline{Q}) \cong \overline{Q}$ (or $\overline{Q}^{op}$) if and only if
    $\mu_\varpi(Q) \cong Q$ (or $Q^{op}$).
\end{itemize}
\end{lem}
\begin{proof}
Part (a) follows from Step 2 in the proof of above theorem that the construction of a prime gluing free quiver (or gluing free quiver) commutates with the quiver mutation. Part $(b)$ follows from $(a)$ and the uniqueness of a prime
gluing free quiver (or gluing free quiver respectively).
\end{proof}
\begin{defn}\label{def: gluing free cluster algebra}
We say a cluster algebra $\A$ prime gluing free (gluing free respectively), if there is a rooted cluster algebra $(Q,\x,\A)$ of $\A$ such that $Q$ is prime gluing free ( gluing free respectively).
\end{defn}
Then it follows from Lemma \ref{main lem in CS} that a cluster algebra is prime gluing free (gluing free respectively), if and only if all of its exchange quivers are prime gluing free (gluing free respectively). We define $\A_{\overline{Q}}$ constructed in Theorem \ref{main thm 1 of pgq} (Remark \ref{rem of gf CA}(b) respectively) as the prime gluing free (gluing free respectively) cluster algebra of $\A_{Q}$.\\

Let $B_{n\times n}$ be an skew-symmetric square matrix, we call a cluster algebra $\A^{pr}$ with the initial matrix $B^{pr}={B\choose I_n}$ a cluster algebra with principal coefficients, where $I_n$ is an identity matrix.

\begin{prop}\label{pgf of PCA}
Any cluster algebra with principal coefficients is prime gluing free.
\end{prop}
\begin{proof}
This is clearly.
\end{proof}
\begin{prop}\label{pgf of UCA}
Any universal geometric cluster algebra is prime gluing free.
\end{prop}
\begin{proof}
Let $(Q,\x,\A_{Q})$ be a rooted cluster algebra such that $\A_{Q}$ is a universal geometric cluster algebra with coefficient group
$\mathbb{P}$. Denote by $B=(b_{ji})_{(m+n)\times n}$ the matrix associated to $Q$. To show that $\A_Q$ is prime gluing free, it is sufficient
to show that $Q$ is prime gluing free.\\

Firstly, we show that all the vertices in $Q$ are prime. Suppose that a frozen vertex $j_1$ of $Q$ is not prime. Denote by
$q=g.c.d(b_{j_11},\cdots,b_{j_1n}) \in \mathbb{Z}^{+}$ the greatest common divisor of the integers in $j_1$ row of $B$. Let $\overline{B}$ be
the matrix obtained from $B$ by replacing the elements $b_{j_1i}$ by $\frac{\displaystyle  b_{j_1i}}{\displaystyle q}$ for any $1\leqslant i
\leqslant n$. Denote by $\overline{Q}$ the ice quiver associated to $\overline{B}$. Then we have a rooted cluster algebra
$(\overline{Q},\overline{\x},\A_{\overline{Q}})$ in $RCA(Q^{ex})$, and we denote its coefficient group by $\overline{\mathbb{P}}$.
Define a group homomorphism $\varphi:\overline{\mathbb{P}} \to \mathbb{P}$ such that $\varphi(\overline{x}_{j_1})={x_{j_1}}^{q}$ and
$\varphi(\overline{x}_j)=x_j$ for the other frozen variables $\overline{x}_j$ in $\overline{\x}$. Then it is not hard to see that $\varphi$ is
a coefficient specialization. On the one hand, since a composition of any two coefficient specialization is also a coefficient specialization,
$\A_{\overline{Q}}$ is a universal geometric cluster algebra with principal part $Q^{ex}$. By the uniqueness of the
coefficient specialization, $\varphi$ is the only coefficient specialization from $\overline{\mathbb{P}}$ to $\mathbb{P}$, but note that the
matrix of $\varphi$ is not invertible. This contradicts the uniqueness of the universal geometric cluster algebra.\\

Secondly, we show that $Q$ is strictly gluing free. Suppose that $Q$ is not strictly gluing free, since $Q$ is prime, this means that there exist two strictly glueable frozen vertices $j_1$ and $j_2$ in $Q$. We define $\varphi$ as a map from $\mathbb{P}$ to $\mathbb{P}$ with
$\varphi(x_{j_1})=x_{j_2}$, $\varphi(x_{j_2})=x_{j_1}$ and $\varphi(x_{j_k})=x_{j_k}$ for the other frozen variables in $\x$. Then $\varphi$
induces a group homomorphism from $\mathbb{P}$ to $\mathbb{P}$ and in fact a coefficient specialization, which is different from the identity
coefficient specialization. This contradicts the uniqueness of the coefficient specialization. Thus $Q$ is gluing free, and we are done.

\end{proof}

\begin{prop}\label{prop of SM pgf}
Let $(S,M)$ be an oriented marked Riemann surface. Assume that $(S,M)$ is not a 4-gon. Then the cluster algebras $\A^{b}(S,M)$ associated to $(S,M)$
with coefficients is prime gluing free.
\end{prop}
\begin{proof}
To prove that $\A^{b}(S,M)$ is prime gluing free, we only need to find a prime gluing free quiver of $\A^{b}(S,M)$.
Assume that $T$ is an ideal triangulation without self-floded triangles, Lemma \ref{lem in Surface} makes sense the assumption. Let $T^{b}$ be
the corresponding triangulation with boundary. Denote by $Q^{b}$ the quiver of $T^{b}$. We unify the arcs in $T^{b}$ with the vertices in
$Q^{b}$. Firstly, since each boundary is contained in a single triangle, there is at most one arrow between its corresponding vertex and any
exchangeable vertex. Thus in $Q^{b}$, each frozen vertex is prime. Secondly, we prove that $\A^{b}(S,M)$ is gluing free. Since there are no
self-floded triangles in $T$, the type I triangle in Figure \ref{puzzle-pieces} is the only possible puzzle piece of $T^{b}$. Let $j$ be a
frozen vertex in $Q^{b}$, assume that $j$ is adjacent to an exchangeable vertex $i$. Without loss of generality, we assume that they are in the
following triangles, where the orientation is clockwise. Then there is an arrow from $j$ to $i$. If there exists a frozen vertex $k$ glueable
with $j$, then by the orientation of the surface, it must be in the following position. We claim that the vertex $ac$ must be frozen. In fact,
if $ac$ is exchangeable, then there is an arrow from $ac$ to $j$. Since $j$ and $k$ are glueable, there is an arrow from $ac$ to $k$. Thus $ac$
and $k$ must be in a triangle of $T^{b}$, which is depicted as $\Delta acd$ in the picture. Therefore $k$ is an exchangeable vertex, and we
have a contradiction. Similarly, $bd$ is also a frozen vertex. Thus $(S,M)$ is a 4-gon, a contradiction to the assumption. Therefore
$\A^{b}(S,M)$ is gluing free. We are done.
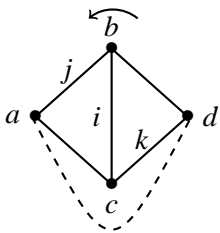
\begin{center}
\begin{tikzpicture}
						\draw[-,thick] (0,0) -- (1,0.9);
						\draw[-,thick] (0,0) -- (1,-0.9);
						\draw[-,thick] (2,0) -- (1,0.9);
						\draw[-,thick] (2,0) -- (1,-0.9);
						\draw[-,thick] (1,0.9) -- (1,-0.9);
                        \node[] () at (0,0) {$\bullet$};
                        \node[] () at (1,0.9) {$\bullet$};
                        \node[] () at (1,-0.9) {$\bullet$};
                        \node[] () at (2,0) {$\bullet$};
                        \node[] () at (0.4,0.6) {$j$};
                        \node[] () at (0.8,0) {$i$};
                        \node[] () at (1.4,-0.3) {$k$};
                        \node[] () at (-0.3,0) {$a$};
                        \node[] () at (1,1.2) {$b$};
                        \node[] () at (1,-1.2) {$c$};
                        \node[] () at (2.3,0) {$d$};
                        \draw[<-,thick] (0.7,1.3) arc (130:50:0.5);
                        \draw[-,thick,dashed] (-0.05,0)..controls(1,-2)..(2.05,0);
\end{tikzpicture}
\end{center}
\end{proof}

\begin{exm}\label{exm:ice quiver of 4-gon}
We consider the $4$-gon with the following triangulation.

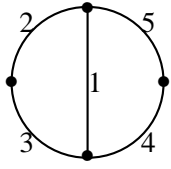
\begin{center}

\begin{tikzpicture}
\node[] (C) at (-3,0)
                       {\begin{tikzpicture}
						\draw[thick] (0,0) circle (1);	
    				    \draw[-,thick] (0,1) -- (0,-1);
                        \node[] () at (0,0.98) {$\bullet$};
                        \node[] () at (0,-0.98) {$\bullet$};
                        \node[] () at (-1,0) {$\bullet$};
                        \node[] () at (1,0) {$\bullet$};
                        \node[] () at (0.1,0) {$1$};
                        \node[] () at (-0.8,0.8) {$2$};
                        \node[] () at (0.8,-0.8) {$4$};
                        \node[] () at (0.8,0.8) {$5$};
                        \node[] () at (-0.8,-0.8) {$3$};
                        \end{tikzpicture}};
\node[] (C) at (3,0) 						
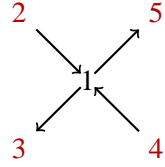
{\begin{tikzpicture}[scale=0.5]
                        \node[] () at (0,0) {$1$};
                        \node[] () at (-1.8,1.8) {$\textcolor[rgb]{0.70,0.00,0.00}{2}$};
                        \node[] () at (-1.8,-1.8) {$\textcolor[rgb]{0.70,0.00,0.00}{3}$};
                        \node[] () at (1.8,-1.8) {$\textcolor[rgb]{0.70,0.00,0.00}{4}$};
                        \node[] () at (1.8,1.8) {$\textcolor[rgb]{0.70,0.00,0.00}{5}$};
						\draw[->,thick] (-1.35,1.35) -- (-0.18,0.18);
						\draw[->,thick] (1.35,-1.35) -- (0.18,-0.18);
						\draw[->,thick] (0.18,0.18) -- (1.35,1.35);
						\draw[->,thick] (-0.18,-0.18) -- (-1.35,-1.35);
						\end{tikzpicture}};
\end{tikzpicture}
\end{center}

Then in its ice quiver $Q^{b}$, the frozen vertices $2$ and $4$, $3$ and $5$ are strictly glueable. Therefore $Q^{b}$ is not gluing free.
\end{exm}

\begin{thm}\label{main thm 2 of pgq}
Let $(Q,\x,\A_{Q})$ be a rooted cluster algebra and $(\overline{Q},\overline{\x},\A_{\overline{Q}})$ be its gluing free rooted cluster algebra.
Let $\{j_{1,1}, \cdots , j_{1,t_1}\}, \cdots , \{j_{s,1}, \cdots , j_{s,t_s}\}$ be the collections of strictly glueable vertices of $Q$. Then
\begin{itemize}
\item[(a)] there is an injective group homomorphism $\Phi$ from $Aut(\A_{Q})$ to $Aut(\A_{\overline{Q}})\times {S_{t_1}}\times\cdots \times
    {S_{t_s}}$;
\item[(b)] there is an injective group homomorphism $\Phi^{+}$ from $Aut^{+}(\A_{Q})$ to $Aut^{+}(\A_{\overline{Q}})\times
    {S_{t_1}}\times\cdots \times {S_{t_s}}$.
\end{itemize}
\end{thm}
\begin{proof}
We use the notations in the Remark \ref{rem of gf CA} (b). Let $\phi: \A_{Q}\to \A_{Q}$ be a cluster automorphism of $\A_{Q}$. Then it induces
a bijection from $\x$ to $\mu_{\varpi}(\x)$ for some admissible sequence $\varpi$, and a permutation $\sigma$ on the vertex set $Q_0$ of $Q$
such that $\phi(x_i)=\mu_{\varpi}(x_{\sigma(i)})$ for any $x_i\in \x$. And $\sigma$ induces a permutation $\sigma'$ on the frozen vertices of
$Q$. Since $\phi$ is an automorphism, $\sigma'$ maps strictly glueable vertices to the strictly glueable ones. Thus $\sigma'$ induces a
permutation on the partition $\{j_{1,1}, \cdots , j_{1,t_1}\}, \cdots , \{j_{s,1}, \cdots , j_{s,t_s}\}$. Therefore $\sigma'$ induces a
permutation $\sigma''$ on the frozen vertices $\{j_{1}, \cdots , j_{s}\}$ of $\overline{Q}$ which maps $j_{k}$ to $j_{\sigma''(k)}$ for each
$1\leqslant k \leqslant s$. Then we can write $\sigma'$ as $(\sigma'',\sigma_1,\cdots,\sigma_s)$ by mapping each $j_{k,l}$ to
$j_{\sigma''(k),\sigma_{k}(l)}$, where $\sigma_k$ is a permutation on $\{1, \cdots , t_k\}$.
Define $\phi':\overline{\x} \to \mu_{\varpi}(\overline{\x})$ as a map by
$\phi'(\overline{x}_i)=\mu_{\varpi}(\overline{x}_{\sigma(i)})$ for any exchangeable vertex $i$ of $\overline{Q}$ and
$\phi'(\overline{x}_{j_k})=\mu_{\varpi}(\overline{x}_{j_{\sigma''(k)}})$ for any $1\leqslant k \leqslant s$. Since $\varphi$ is a cluster
automorphism of $\A_{Q}$, by Proposition \ref{Prop of CA} (a), it maps $Q$ to an ice quiver which is isomorphic to itself or to an opposite quiver. Therefore by using Proposition \ref{Prop of CA} (a) again, Lemma \ref{main lem in CS}(b)
shows that $\phi'$ induces a cluster automorphism $\phi''$ of $\A_{\overline{Q}}$. We define $(\phi'',\sigma_1,\cdots,\sigma_s)$ as an element
$Aut(\A_{\overline{Q}})\times {S_{t_1}}\times\cdots \times {S_{t_s}}$. Then it is not hard to see that $\Phi: \phi \to
(\phi'',\sigma_1,\cdots,\sigma_s)$ is an injective group homomorphism. The second group homomorphism follows from the first one and Lemma \ref{main lem in CS}(b).
\end{proof}

\begin{rem}\label{remark gluing free CA}
Note that a permutation of strictly glueable vertices of $Q$ induces an automorphism of $Q$, then it induces an automorphism of $\A_{Q}$. Thus
$\Phi$ is surjective onto ${S_{t_1}}\times\cdots \times {S_{t_s}}$. However, $\Phi$ may not be a surjection, see in the following example. We
consider the following ice quiver $Q$ with frozen vertices $3$, $4$ and $5$. Then its gluing free quiver is $\overline{Q}$, where $3$ and $4$
are frozen vertices.
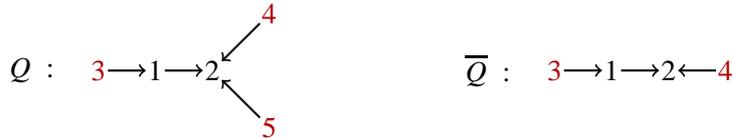
\begin{center}

\begin{tikzpicture}
\node[] (C) at (-5,0) 						{\begin{tikzpicture}[scale=0.5]
                        \node[] () at (0,0) {$Q~:$};
                        					\end{tikzpicture}};
\node[] (C) at (-3,0) 						{\begin{tikzpicture}[scale=0.5]
                        \node[] () at (-3,0) {$\textcolor[rgb]{0.70,0.00,0.00}{3}$};
                        \node[] () at (-1.5,0) {$1$};
                        \node[] () at (0,0) {$2$};
                        \node[] () at (1.5,1.5) {$\textcolor[rgb]{0.70,0.00,0.00}{4}$};
                        \node[] () at (1.5,-1.5) {$\textcolor[rgb]{0.70,0.00,0.00}{5}$};
						\draw[->,thick] (1.25,-1.25) -- (0.25,-0.25);
						\draw[<-,thick] (0.25,0.25) -- (1.25,1.25);
						\draw[->,thick] (-2.75,0) -- (-1.75,0);
						\draw[->,thick] (-1.25,0) -- (-0.25,0);
						\end{tikzpicture}};
\node[] (C) at (1,0) 						{\begin{tikzpicture}[scale=0.5]
                        \node[] () at (0,0) {$\overline{Q}~:$};
                        					\end{tikzpicture}};
\node[] (C) at (3,0) 						{\begin{tikzpicture}[scale=0.5]
                        \node[] () at (-3,0) {$\textcolor[rgb]{0.70,0.00,0.00}{3}$};
                        \node[] () at (-1.5,0) {$1$};
                        \node[] () at (0,0) {$2$};
                        \node[] () at (1.5,0) {$\textcolor[rgb]{0.70,0.00,0.00}{4}$};
						\draw[<-,thick] (0.25,0) -- (1.25,0);
						\draw[->,thick] (-2.75,0) -- (-1.75,0);
						\draw[->,thick] (-1.25,0) -- (-0.25,0);
						\end{tikzpicture}};
\end{tikzpicture}
\end{center}
Then $\mu_{1}\mu_{2}(\overline{Q})\cong {\overline{Q}}^{op}$, and it induces a opposite cluster automorphism $\phi$ of
$\A_{\overline{Q}}$. However $\mu_{1}\mu_{2}(Q)$ is neither isomorphic to $Q$ nor to $Q^{op}$. Thus $\phi$ is not an image of $\Phi$.
\end{rem}

\subsection{Exchange graphs}\label{section exchange graph}
Recall from \cite{FZ02} and \cite{FZ07} that for a cluster algebra $\A_Q$, where $Q$ is an ice quiver with $n$ exchangeable vertices, the exchange graph $E_Q$ of $\A_Q$ is a $n$-regular graph whose vertices are the seeds of $\A_Q$ and two vertices are
connected by an edge labeled by $x_i$, $1\leqslant i\leqslant n$, if and only if the corresponding seeds are adjacent in direction $x_i$. In our settings, the clusters determine the quivers, thus the vertices of $E_Q$ are in fact the clusters of $\A_Q$. It is proved in \cite{CKLP13}(Theorem 4.6) that $E_Q$ only depends on the principal part $Q^{ex}$, that is, there is an isomorphism from $E_Q$ to $E_{Q^{ex}}$. The correspondence on the vertices is given by the specialization $S_E$: for each vertex $\x$ on $E_Q$, $S_E(\x)$ is obtained from $\x$ by specializing any frozen variable in the exchangeable variables to integer $1$, and forgetting the frozen part of $\x$.\\

Now we introduce the following notion:
\begin{defn}
An automorphism of $E_Q$ is an automorphism of $E_Q$ as a graph, that is, a permutation $\sigma$ of the vertex set, such that the pair of
vertices $(u,v)$ forms an edge if and only if the pair $(\sigma(u),\sigma(v))$ also forms an edge.
\end{defn}

Then an automorphism of $E_Q$ gives a bijection on the vertices and on the edges of $E_Q$, and maps adjacent vertices to adjacent ones. It is
clearly that the natural composition of two automorphisms of $E_Q$ is again an automorphism of $E_Q$. We define an {\it automorphism group
$Aut(E_Q)$} of $E_Q$ as a group consisting of automorphisms of $E_Q$ with compositions of automorphisms as multiplications. Then the
isomorphism $S_E: E_Q \to E_{Q^{ex}}$ induces a group isomorphism from
$Aut(E_Q)$ to $Aut(E_{Q^{ex}})$, we also denote it by $S_E$.

\begin{exm}\label{Aut(E) of A3 type}
We consider the cluster algebra of type $A_3$ with initial seed $(Q,\{x_1,x_2,x_3\})$, where $Q$ is
$\xymatrix{1\ar[r]&2\ar[r]&3}.$ Then its exchange graph $E_{Q}$ is depicted in Figure \ref{automorphisms of exchange graph of type $A_3$}, where the vertex $O$ represents the seed $(Q,\{x_1,x_2,x_3\})$. Note that there are three quadrilaterals and six pentagons in $E_{Q}$. Since an automorphism of a graph maps vertices to vertices, edges to edges and thus faces to faces, it is not hard to see that the automorphism group $Aut(E_{Q})$ is generated by $\tau$ and
$\sigma$ showed in Figure \ref{automorphisms of exchange graph of type $A_3$}. The automorphism $\tau$ is given by the dashed lines , maps red
edges to red ones and maps green edges to green ones. The automorphism $\sigma$ is the reflection along the central axis. Note that the order
of $\tau$ is $6$, the order of $\sigma$ is $2$ and the only relation which they satisfy is $\sigma\tau=\tau^{5}\sigma$. Thus $Aut(E_{Q})$ is
isomorphic to the dihedral group $D_6$.

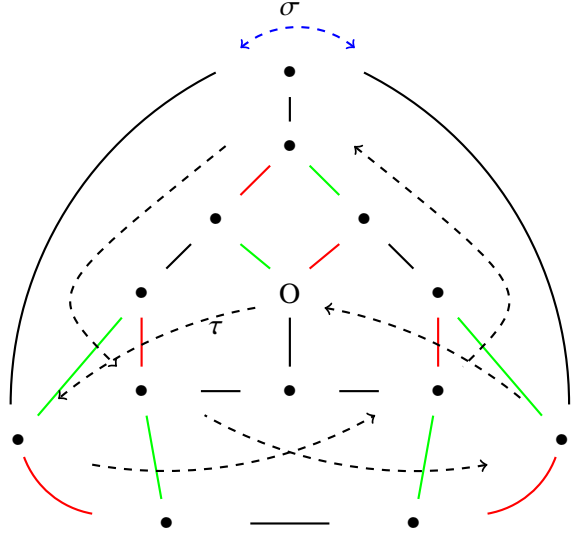
\begin{figure}
\begin{center}
{\begin{tikzpicture}[scale=0.65]
\node[] (C) at (0,0)
						{\text{O}};
\node[] (C) at (-1.5,1.5)
						{$\bullet$};
\node[] (C) at (1.5,1.5)
						{$\bullet$};
\node[] (C) at (-3,0)
						{$\bullet$};
\node[] (C) at (3,0)
						{$\bullet$};
\node[] (C) at (0,-2)
						{$\bullet$};
\node[] (C) at (0,3)
						{$\bullet$};
\node[] (C) at (0,4.5)
						{$\bullet$};
\node[] (C) at (-3,-2)
						{$\bullet$};
\node[] (C) at (3,-2)
						{$\bullet$};
\node[] (C) at (-5.5,-3)
						{$\bullet$};
\node[] (C) at (5.5,-3)
						{$\bullet$};
\node[] (C) at (-2.5,-4.7)
						{$\bullet$};
\node[] (C) at (2.5,-4.7)
						{$\bullet$};
\node[] (C) at (0,5.8)
						{$\sigma$};
\node[] (C) at (-1.5,-0.7)
						{$\tau$};

\draw[thick] (0,3.5) -- (0,4);
\draw[red,thick] (-0.4,2.6) -- (-1,2);
\draw[green,thick] (0.4,2.6) -- (1,2);
\draw[green,thick] (-0.4,0.5) -- (-1,1);
\draw[red,thick] (0.4,0.5) -- (1,1);
\draw[thick] (-2.5,0.5) -- (-2,1);
\draw[thick] (2.5,0.5) -- (2,1);
\draw[thick] (0,-0.5) -- (0,-1.5);
\draw[red,thick] (-3,-0.5) -- (-3,-1.5);
\draw[red,thick] (3,-0.5) -- (3,-1.5);
\draw[thick] (-1.8,-2) -- (-1,-2);
\draw[thick] (1.8,-2) -- (1,-2);
\draw[thick] (-0.8,-4.7) -- (0.8,-4.7);
\draw[green,thick] (-2.9,-2.5) -- (-2.6,-4.2);
\draw[green,thick] (2.9,-2.5) -- (2.6,-4.2);
\draw[green,thick] (-5.1,-2.5) -- (-3.4,-0.5);
\draw[green,thick] (5.1,-2.5) -- (3.4,-0.5);
\draw[thick] (1.5,4.5) arc (63:0:7.6);
\draw[thick] (-1.5,4.5) arc (-63:0:-7.6);
\draw[red,thick] (4,-4.5) arc (-80:-20:1.8);
\draw[red,thick] (-4,-4.5) arc (80:20:-1.8);
\draw[<->,blue,thick,dashed] (-1,5) arc (130:50:1.8);

\draw[->,thick,dashed] (-0.7,-0.3) arc (100:130:8.5);
\draw[->,thick,dashed] (-4.0,-3.5) arc (-100:-60:8.5);
\draw[<-,thick,dashed] (4.0,-3.5) arc (100:60:-8.5);
\draw[<-,thick,dashed] (1.3,3.0)..controls(5,0)..(3.5,-1.5);
\draw[<-,thick,dashed] (0.7,-0.3) arc (-100:-130:-8.5);
\draw[->,thick,dashed] (-1.3,3.0)..controls(-5,0)..(-3.5,-1.5);
\end{tikzpicture}}
\end{center}
\begin{center}
\caption{The automorphisms of exchange graph of type $A_3$}
\label{automorphisms of exchange graph of type $A_3$}
\end{center}
\end{figure}

\end{exm}

\begin{thm}\label{AA subgroup of AE}
Let $Q$ be an ice quiver with at least two exchangeable vertices.
\begin{itemize}
\item[(a)] If $Q$ is gluing free, then the cluster automorphism group $Aut(\A_Q)$ is a subgroup of $Aut(E_Q)$.
\item[(b)] If $Q$ is not gluing free, denoted by $\{j_{1,1}, \cdots , j_{1,t_1}\}, \cdots , \{j_{s,1}, \cdots , j_{s,t_s}\}$ the collections of strictly glueable vertices
    of $Q$, then the cluster automorphism group $Aut(\A_Q)$ is a subgroup of $Aut(E_Q)\times {S_{t_1}}\times\cdots \times {S_{t_s}}$.
\end{itemize}
\end{thm}
\begin{proof}
(a) Let $f$ be an automorphism of $\A_Q$, then Proposition \ref{Prop of CA}(3) yields that $f$ gives rise to a map $f'$ from the vertices of
$E_Q$ to the vertices of $E_Q$. In fact $f'$ is a bijection since $f$ is an automorphism. For any vertex $\x$ and its adjacent vertex
$\x'=\x\setminus \{x\}\sqcup \{\mu_x(x)\}$ in $E_Q$, we have $f(\mu_x(x))=\mu_{f(x)}(f(x))$ since $f$ is a cluster automorphism. Thus we have
$f'(\x')=f'(\x)\setminus\{f(x)\}\sqcup\{\mu_{f(x)}(f(x)\}$. Therefore the vertices $f'(\x)$ and $f'(\x')$ are adjacent in $E_Q$. We have proved
that $f'$ is an automorphism of $E_Q$, and it is not hard to see that this yields a group homomorphism from $Aut(\A_Q)$ to $Aut(E_Q)$. For the
injectivity, if $f'$ is an identity on $E_Q$, then $f$ fix any cluster, and it is not hard to see that it fix any exchangeable variable. On the
other hand, since $Q$ is gluing free and $Q$ has at least two exchangeable vertices, any permutation of frozen variables does not induce an automorphism of $\A_Q$. Thus $f$ fixes any cluster variables, and it is an identity on $\A_Q$.
Therefore $Aut(\A_Q)$ is a subgroup of $Aut(E_Q)$.\\

(b) Let $\overline{Q}$ be the gluing free quiver of $Q$, then it is gluing free. Therefore, by (a), $Aut(\A_{\overline{Q}})$ is a subgroup of
$Aut(E_{\overline{Q}})$. Moreover, since $Aut(E_Q) \cong Aut(E_{\overline{Q}})$ and $Aut(\A_Q)$ is a subgroup of $Aut(\A_{\overline{Q}})\times
{S_{t_1}}\times\cdots \times {S_{t_s}}$ by Theorem \ref{main thm 2 of pgq}, $Aut(\A_Q)$ is a subgroup of $Aut(E_Q)\times {S_{t_1}}\times\cdots
\times {S_{t_s}}$.
\end{proof}

\begin{exm}\label{exm:group of 4-gon}
If $Q$ is an ice quiver with a single exchangeable vertex, the results in Theorem \ref{AA subgroup of AE} may be not true. We consider the cluster algebra $\A_{Q^{b}}$ of a $4$-gon with the initial triangulation given in Example \ref{exm:ice quiver of 4-gon}.
Then the exchange graph of $\A_{Q^{b}}$ is a line segment with two vertices and its automorphism group $Aut(E_{Q^{b}})$ is isomorphic to $S_2$.
$$						{\begin{tikzpicture}
						\node[] () at (-0.7,0) {$\bullet$};
						\node[] () at (0.7,0) {$\bullet$};
						\draw[-,thick] (-0.7,0) -- (0.7,0);
						\end{tikzpicture}}$$
On the other hand, $Aut(\A_{Q^{b}})\cong <\sigma_{x_1\mu_1(x_{1})}> \times <\sigma_{x_2x_{4}}> \times <\sigma_{x_3x_{5}}> \times
<\sigma_{x_2x_{5}}\sigma_{x_3x_{4}}>  \cong S_2\times S_2\times S_2\times S_2$, where we denote by $\sigma_{x_ix_j}$
($\sigma_{x_1\mu_1(x_{1})}$ respectively) the automorphism of $\A_{Q^{b}}$ induced by the permutation of cluster variables $x_i$ and $x_j$
($x_1$ and $\mu_1(x_{1})$ respectively). Then $Aut(\A_{Q^{b}})\cong S_2\times S_2\times S_2\times S_2$ is not a subgroup of
$Aut(E_{Q^{b}})\times S_2 \times S_2\cong S_2\times S_2\times S_2$.
\end{exm}

\begin{exm}\label{AA=AE for A3 type}
We consider the cluster algebra of type $A_3$ with initial seed $(Q,\{x_1,x_2,x_3\})$, where $Q$ is
$\xymatrix{1\ar[r]&2\ar[r]&3}.$ As showed in Example \ref{Aut(E) of A3 type}, $Aut(E_{Q})$ is isomorphic to the dihedral group $D_6$ generated by $\tau$ and $\sigma$.
Here $\tau$ induces a bijection on the cluster variables of $\A_Q$, and we also denote it by $\tau$. In fact $\tau$
is a cluster automorphism which is induced by the AR-translation on the cluster category of the algebra $kQ$ (see Section 3 of \cite{ASS12} for more details), it is given
by:
$$\tau: \left\{\begin{array}{rcl}
				x_1 & \mapsto & \frac{1+x_2}{x_1} \\
				x_2 & \mapsto & \frac{x_1+x_3+x_2 x_3}{x_1 x_2} \\
				x_3 & \mapsto & \frac{(1+x_1)(x_1+x_3)}{x_1 x_2 x_3}~~.
			\end{array}\right.$$
On the other hand, the automorphism $\sigma$ of $E_Q$ corresponds to the opposite automorphism of $\A_Q$ given by exchanging the cluster
variables $x_1$ and $x_2$. Therefore $Aut(E_{Q})\subseteq Aut(\A_{Q})$. Note that from Theorem \ref{AA subgroup of AE}, $Aut(\A_{Q})$ is a subgroup of $Aut(E_{Q})$, thus $Aut(\A_{Q})\cong
Aut(E_{Q}) \cong D_6$.
\end{exm}
\begin{rem}
In the subsequent paper \cite{CZb15} we prove that for a coefficient free cluster algebra $\A_Q$, the two groups $Aut(\A_{Q})$ and $Aut(E_{Q})$ are isomorphic with each other, if $\mathcal{A}$ is of finite type, excepting types of rank two and type $F_4$, or $\mathcal{A}$ is of skew-symmetric finite mutation type. For a cluster algebra of type $F_4$ and all the non-skew-symmetric cluster algebra of rank two, $Aut(\A_{Q})$ is a proper subgroup of $Aut(E_{Q})$.
\end{rem}
\begin{exm}\label{AA is proper of AE for A3 type with coefficients}
Let $Q'$ be the following quiver, where vertices $4$ and $5$ are frozen.
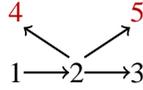
\begin{center}
\begin{tikzpicture}
						\node[] () at (-0.8,0) {$1$};
						\node[] () at (0,0) {$2$};
						\draw[->,thick] (-0.7,0) -- (-0.1,0);
                        \node[] () at (0.8,0) {$3$};
						\draw[->,thick] (0.1,0) -- (0.7,0);
						\draw[->,thick] (-0.1,0.2) -- (-0.7,0.6);
						\draw[->,thick] (0.1,0.2) -- (0.7,0.6);
                        \node[] () at (-0.8,0.8) {$\textcolor[rgb]{0.70,0.00,0.00}{4}$};
                        \node[] () at (0.8,0.8) {$\textcolor[rgb]{0.70,0.00,0.00}{5}$};
\end{tikzpicture}
\end{center}
The seed $(Q',\{x_1,x_2,{x_3},{x_{4}},x_5\})$ defines a cluster algebra $\A_{Q'}$. We consider the cluster automorphism group
$Aut(\A_{Q'})$. Note that the vertices $4$ and $5$ are strictly glueable, thus exchanging the corresponding cluster variables $x_4$ and $x_5$
induces a cluster automorphism $\delta$ of $Aut(\A_{Q'})$, whose order is two, thus $\{1, \delta\} \cong S_2$ is a subgroup of $Aut(\A_{Q'})$.
On the other hand, we consider the gluing free quiver of $Q'$:

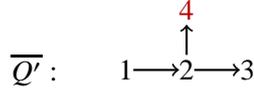
\begin{center}
\begin{tikzpicture}  	\node[] () at (-2,0) {$\overline{Q'} :$};
						\node[] () at (-0.8,0) {$1$};
						\node[] () at (0,0) {$2$};
						\draw[->,thick] (-0.7,0) -- (-0.1,0);
                        \node[] () at (0.8,0) {$3$};
						\draw[->,thick] (0.1,0) -- (0.7,0);
						\draw[->,thick] (0,0.2) -- (0,0.6);
                        \node[] () at (0,0.8) {$\textcolor[rgb]{0.70,0.00,0.00}{4}$};
\end{tikzpicture}
\end{center}
The exchange graph of $\A_{\overline{Q'}}$ is depicted in Figure \ref{automorphisms of cluster algebra of type $A_3$ with a coefficient}, which is isomorphic to the
one in Figure \ref{automorphisms of exchange graph of type $A_3$}.
An automorphism in $Aut(\A_{\overline{Q'}})$ maps $\overline{Q'}$ to a quiver isomorphic to $\overline{Q'}$ or to $\overline{Q'}^{op}$. A direct compute shows that there are two seeds satisfy the above property: $\mu_{1}\mu_{3}(\overline{Q'})\cong \overline{Q'}$, $\mu_{3}\mu_{2}\mu_{1}\mu_{2}(\overline{Q'})\cong \overline{Q'}^{op}$. Then the automorphisms of
$\A_{\overline{Q'}}$ is showed in Figure \ref{automorphisms of cluster algebra of type $A_3$ with a coefficient} (compare with Figure
\ref{automorphisms of exchange graph of type $A_3$}), where $\tau$ and $\sigma$ are algebra homomorphisms determined by:
$$\tau: \left\{\begin{array}{rcl}
				x_1 & \mapsto & \frac{1+x_2}{x_1} \\
				x_2 & \mapsto & \frac{x_1+x_3 x_4+x_2 x_3 x_4}{x_1 x_2} \\
				x_3 & \mapsto & \frac{(1+x_1)(x_1+x_3 x_4)}{x_1 x_2 x_3}\\
				x_4 & \mapsto & x_4
			\end{array}\right.$$
and
$$\sigma: \left\{\begin{array}{rcl}
				x_1 & \mapsto & x_3 \\
				x_2 & \mapsto & x_2 \\
				x_3 & \mapsto & x_1\\
				x_4 & \mapsto & x_4.
			\end{array}\right.$$
Then $Aut(\A_{\overline{Q'}})=\{1, \tau^3, \tau\sigma, \tau^4\sigma\} \cong K_4$ is a proper subgroup of $Aut(E_{\overline{Q}})$, which is isomorphic
to $D_6$. Finally we have $Aut(\A_{Q'}) \cong Aut(\A_{\overline{Q'}}) \times \{1, \delta\} \cong S_2 \times S_2 \times S_2$.

\begin{figure}
\begin{center}
{\begin{tikzpicture}[scale=0.65]
\node[] (C) at (0,0)
						{\text{O}};
\node[] (C) at (-1.5,1.5)
						{$\bullet$};
\node[] (C) at (1.5,1.5)
						{$\bullet$};
\node[] (C) at (-3,0)
						{$\bullet$};
\node[] (C) at (3,0)
						{$\bullet$};
\node[] (C) at (0,-2)
						{$\bullet$};
\node[] (C) at (0,3)
						{$\bullet$};
\node[] (C) at (0,4.5)
						{$\bullet$};
\node[] (C) at (-3,-2)
						{$\bullet$};
\node[] (C) at (3,-2)
						{$\bullet$};
\node[] (C) at (-5.5,-3)
						{$\bullet$};
\node[] (C) at (5.5,-3)
						{$\bullet$};
\node[] (C) at (-2.5,-4.7)
						{$\bullet$};
\node[] (C) at (2.5,-4.7)
						{$\bullet$};
\node[] (C) at (-1,-1.3)
						{\qihao $\tau^{4}\sigma$};
\node[] (C) at (-2.6,-0.6)
						{\qihao $\tau\sigma$};
\node[] (C) at (0.3,1.6)
						{\qihao $\tau^{3}$};

\draw[thick] (0,3.5) -- (0,4);
\draw[red,thick] (-0.4,2.6) -- (-1,2);
\draw[green,thick] (0.4,2.6) -- (1,2);
\draw[green,thick] (-0.4,0.5) -- (-1,1);
\draw[red,thick] (0.4,0.5) -- (1,1);
\draw[thick] (-2.5,0.5) -- (-2,1);
\draw[thick] (2.5,0.5) -- (2,1);
\draw[thick] (0,-0.5) -- (0,-1.5);
\draw[red,thick] (-3,-0.5) -- (-3,-1.5);
\draw[red,thick] (3,-0.5) -- (3,-1.5);
\draw[thick] (-1.8,-2) -- (-1,-2);
\draw[thick] (1.8,-2) -- (1,-2);
\draw[thick] (-0.8,-4.7) -- (0.8,-4.7);
\draw[green,thick] (-2.9,-2.5) -- (-2.6,-4.2);
\draw[green,thick] (2.9,-2.5) -- (2.6,-4.2);
\draw[green,thick] (-5.1,-2.5) -- (-3.4,-0.5);
\draw[green,thick] (5.1,-2.5) -- (3.4,-0.5);
\draw[thick] (1.5,4.5) arc (63:0:7.6);
\draw[thick] (-1.5,4.5) arc (-63:0:-7.6);
\draw[red,thick] (4,-4.5) arc (-80:-20:1.8);
\draw[red,thick] (-4,-4.5) arc (80:20:-1.8);

\draw[<->,thick,dashed] (-0.7,0) arc (110:130:14);
\draw[<->,thick,dashed] (-0.3,-0.25) arc (135:115:-8);
\draw[<->,thick,dashed] (0,0.6)-- (0, 2.6);
\end{tikzpicture}}
\end{center}
\begin{center}
\caption{The automorphisms of a cluster algebra of type $A_3$ with a coefficient}
\label{automorphisms of cluster algebra of type $A_3$ with a coefficient}
\end{center}
\end{figure}
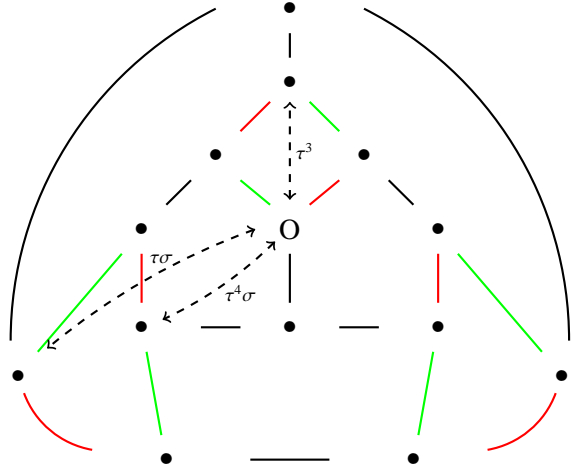

\end{exm}

\subsection{Main results}\label{section main results}
\begin{thm}\label{CCA subgroup of CA}
Let $Q'$ be an ice quiver with principal part $Q$. Assume that there are at least two vertices in $Q$.
\begin{itemize}
\item[(a)] If $Q'$ is gluing free, then $Aut(\A_{Q'})$ is a subgroup of $Aut(\A_Q)$, and $Aut^+(\A_{Q'})$ is a subgroup of $Aut^+(\A_Q)$.
\item[(b)] If $Q'$ is not gluing free, denoted by $\{j_{1,1}, \cdots , j_{1,t_1}\}, \cdots , \{j_{s,1}, \cdots , j_{s,t_s}\}$ the collections of strictly glueable vertices
    of $Q'$. then $Aut(\A_{Q'})$ is a subgroup of $Aut(\A_Q)\times {S_{t_1}}\times\cdots \times {S_{t_s}}$, and $Aut^+(\A_{Q'})$ is a
    subgroup of $Aut^+(\A_Q)\times {S_{t_1}}\times\cdots \times {S_{t_s}}$.
\end{itemize}
\end{thm}
\begin{proof}
(a) We define a group homomorphism $S_{\A}$ from $Aut(\A_{Q'})$ to $Aut(\A_Q)$, and prove the following diagram of group
homomorphisms commutes :
$$\xymatrix{Aut(E_{Q'})\ar[r]^{S_E}&Aut(E_Q)\\
Aut(\A_{Q'})\ar[u]_{i'}\ar[r]^{S_{\A}}&Aut(\A_{Q})\ar[u]_{i}~,}$$
where $i$ and $i'$ are injections by Theorem \ref{AA subgroup of AE}(a), and $S_E$ is the isomorphism given in subsection \ref{section exchange graph}. Then the morphism $i\circ S_{\A}=S_E \circ i'$ is injective, thus $S_{\A}$
is injective, that is, $Aut(\A_{Q'})$ is a subgroup of $Aut(\A_Q)$.\\

We consider the rooted cluster algebra $(Q',\x',\A_{Q'})$ and the rooted cluster algebra $(Q,\x,\A_Q)$. Assume that $\x=\{x_1,x_2,\cdot
\cdot \cdot, x_n\}$ and $\x'=\{x_1,x_2,\cdot \cdot \cdot, x_{n}, x_{n+1},\cdot \cdot \cdot, x_{n+m}\}$, where $x_{n+1},\cdot \cdot \cdot,
x_{n+m}$ are frozen variables of $\A_{Q'}$. Define $\Phi: \A_{Q'} \to \A_Q$ as the specialization, that is, specializing frozen variables in
$\A_{Q'}$ to the integer $1$. It is clearly that $\Phi(\x')=\x\sqcup \{1\}$.
It is proved in \cite{CZ14}(Proposition 2.39) that $\Phi$ is a surjective rooted cluster morphism from $(Q',\x',\A_{Q'})$ to $(Q,\x,\A_Q)$.
Let $f$ be an element in $Aut(\A_Q)$. We define $S_\A(f)=\Phi f \Phi^{-1}: \A_Q \to \A_Q$, then it is a ring homomorphism. To prove that $S_\A(f)$ is an cluster automorphism on $\A_Q$, we show that it satisfies the conditions in Proposition \ref{Prop of CA}(2). This easily follows from that $f$ is a cluster automorphism. Moreover, it is not hard to see that $S_\A$ is a group homomorphism. Finally, the above diagram is commutative since $S_E$ and $S_\A$ are both defined by
specializations. Then $Aut(\A_{Q'})$ is a subgroup of $Aut(\A_Q)$. The injection $S_\A: Aut(\A_{Q'})\to Aut(\A_Q)$ induces an injection from
$Aut^+(\A_{Q'})$ to $Aut^+(\A_Q)$. Thus we have done.\\

(b) This follows from Theorem \ref{AA subgroup of AE}(b) and (a).
\end{proof}

\begin{cor}\label{corollary of UCA}
\begin{itemize}
\item[(a)] For a cluster algebra with principal coefficients, excepting type $A_1$, its cluster automorphism group is a subgroup of the cluster automorphism group of its principal part cluster algebra.
\item[(b)] For a universal geometric cluster algebra, excepting type $A_1$, its cluster automorphism group is a subgroup of the cluster automorphism group of its principal part cluster algebra.
\end{itemize}
\end{cor}
\begin{proof}
It is prove in Proposition \ref{pgf of PCA} and Proposition \ref{pgf of UCA} that these two kinds of cluster algebras are prime gluing free,
and thus gluing free. Then the result follows from above theorem.
\end{proof}
\begin{exm}\label{universal subgroup}
We consider the cluster algebra $\A_{Q}$ of type $A_2$ and the two of its universal geometric cluster algebras $\A_{Q'}$ and $\A_{Q''}$ in Example
\ref{two univ algs}. Then the cluster automorphism groups $Aut(\A_{Q'})$ and $Aut(\A_{Q''})$ are both subgroups of $Aut(\A_{Q})$ which is
isomorphic to $D_5$. A straightforward compute shows that $Aut(\A_{Q'})\cong D_5$ and $Aut(\A_{Q''})=\{id\}$. We prove in the subsequent paper \cite{CZ15a} that $Aut(\A)$ is always isomorphic to $Aut(\A^{univ})$ for a cluster algebra $\A$ of finite type and its $FZ$-universal cluster algebra $\A^{univ}$.
\end{exm}

For a gluing free cluster algebra, its cluster automorphism group may be a proper subgroup of the cluster automorphism group of its principal part cluster algebra (see Example \ref{AA is proper of AE for A3 type with coefficients} and Example \ref{universal subgroup}). We show in the following theorem that for most cluster algebras associated to surfaces, the coefficients do not affect the cluster automorphism groups.

\begin{thm}\label{theorem SCA}
Let $(S,M)$ be an oriented marked Riemann surface satisfies the Assumption 1, and not be a 4-gon. Let $\A(S,M)$ and $\A^{b}(S,M)$ be the cluster
algebras associated to $(S,M)$ without coefficients and with coefficients respectively. Then there are isomorphisms $Aut^+(\A(S,M))\cong
Aut^+(\A^{b}(S,M))$ and $Aut(\A(S,M))\cong Aut(\A^{b}(S,M))$.
\end{thm}
\begin{proof}
For the surface satisfies the Assumption 1, its marked mapping class group $\mmg$ are isomorphic to the direct cluster automorphism group
$Aut^+(\A(S,M))$. On the one hand, it follows from Proposition \ref{prop of SM pgf} and Theorem \ref{CCA subgroup of CA} that
$Aut^+(\A^{b}(S,M))$ is a subgroup of $Aut^+(\A(S,M))$, and thus a subgroup of $\mmg$. On the other hand, it is clearly that an element in the
marked mapping class group $\mmg$ maps the boundaries to the boundaries and maintains the combinatorics of the tagged triangulations with
boundaries. Thus it not only induces a cluster automorphism of $\A(S,M)$, but also induces a cluster automorphism of $\A^{b}(S,M)$. Therefore
the three groups $Aut^+(\A(S,M))$, $\mmg$ and $Aut^+(\A^{b}(S,M))$ are isomorphic. For the second isomorphism, note that if
$Aut^+(\A(S,M))\cong Aut(\A(S,M))$ then $Aut(\A(S,M)) \cong Aut(\A^{b}(S,M))$. Now we assume that $Aut^+(\A(S,M))$ is a proper subgroup of $
Aut(\A(S,M))$, then by Lemma \ref{index of direct auto} its index is $2$. Thus there exist two triangulations $T$ and $T'$ of $(S,M)$ such that
$Q_{T}\cong Q_{T'}^{op}$. Then by the orientation and the combinatorics of the surface, we have $Q_{T^{b}}\cong Q_{T'^{b}}^{op}$. Thus the
index of $Aut^{+}(\A^{b}(S,M))$ in $Aut(\A^{b}(S,M))$ is also 2. On the other hand, there is an isomorphism $Aut^+(\A(S,M))\cong
Aut^+(\A^{b}(S,M))$ from the first step, and $Aut(\A^{b}(S,M))$ is a subgroup of $Aut(\A(S,M))$ by Theorem \ref{CCA subgroup of CA}. Thus we
have an isomorphism between $Aut(\A(S,M))$ and $Aut(\A^{b}(S,M))$.
\end{proof}

\begin{thm}\label{thm:principal coefficients}
Let $Q$ be a quiver with at least two vertices. Let $\A^{pr}$ be the principal coefficient cluster algebra of $Q$. Then $Aut(\A^{pr})\cong Aut^{+}(\A^{pr})\cong Aut(B)\cong Aut(B^{pr})$.
\end{thm}
\begin{proof}
Firstly,  note that $Aut(B)\cong Aut(B^{pr})$, and an automorphism of $B$, or equivalently, an automorphism of $B^{pr}$ induces an automorphism of $\A^{pr}$ by exchanging the initial cluster variables of $\A^{pr}$. Thus $Aut(B)\cong Aut(B^{pr})\subseteq Aut(\A^{pr})$. Secondly, it is proved in \cite{BDP14} (see Proposition 2.10 and Corollary 2.12) that in seeds on exchange graph $E_{\A^{pr}}$, excepting $Q^{pr}$ itself, there exist no ice quivers which are isomorphic to $Q^{pr}$, and there exist no ice quivers which are isomorphic to the opposite quiver of $Q^{pr}$. Thus any automorphism $\sigma\in Aut(\A^{pr})$ is direct and must map $Q^{pr}$ to itself. Therefore $Aut(\A^{pr})\cong Aut^{+}(\A^{pr})\cong Aut(B)\cong Aut(B^{pr})$.
\end{proof}

\section*{Acknowledgements}
The authors wish to thank Dong Yang, Wuzhong Yang, Jie Zhang and Yu Zhou for helpful discussions on the topic, and thank Bangming Deng, Yang Han, Yanan Lin, Liangang Peng and Jie Xiao for valuable comments. The first author wish to thank Thomas Br\"ustle for useful discussions during a conference \cite{B14} in August 2014.

\end{document}